\documentclass[12pt,reqno]{amsart}
\usepackage[margin=1in]{geometry}
\usepackage{amsmath,amssymb,amsthm,graphicx,amsxtra, setspace}
\usepackage[utf8]{inputenc}
\usepackage{mathrsfs}
\usepackage{hyperref}
\usepackage{upgreek}
\usepackage{dsfont}
\usepackage{mathtools}
\usepackage{xcolor}
\usepackage{cite}
\allowdisplaybreaks

\usepackage[pagewise]{lineno}

\usepackage{graphicx,eurosym}
\usepackage{hyperref}
\usepackage{mathtools}

\usepackage[cyr]{aeguill}

\colorlet{darkblue}{blue!50!black}

\hypersetup{
	colorlinks,%
	citecolor=blue,%
	filecolor=red,%
	linkcolor=red,%
	urlcolor=blue,%
	pdfnewwindow=true,%
	pdfstartview={FitH}
}

\newtheorem{theorem}{Theorem}[section]
\newtheorem{lemma}[theorem]{Lemma}

\newtheorem{remark}[theorem]{Remark}

\let\originalleft\left
\let\originalright\right
\renewcommand{\left}{\mathopen{}\mathclose\bgroup\originalleft}
\renewcommand{\right}{\aftergroup\egroup\originalright}

\renewcommand{\d}{\/\mathrm{d}\/}

\def\w{\textbf{W}^{\varepsilon}_{{\theta}^{\varepsilon}}}

\def\A{\mathrm{A}}

\def\C{\mathrm{C}}
\def\f{\boldsymbol{f}}

\def\D{\mathrm{D}}
\def\y{\boldsymbol{y}}

\def\k{\boldsymbol{k}}
\def\g{\boldsymbol{g}}

\def\h{\boldsymbol{h}}
\def\z{\boldsymbol{z}}
\def\v{\boldsymbol{v}}
\def\V{\mathbb{v}}
\def\w{\boldsymbol{w}}

\def\no{\nonumber}
\def\V{\mathbb{V}}
\def\wi{\widetilde}

\def\u{\mathrm{U}}

\def\vr{\boldsymbol{\varphi}}
\def\u{\boldsymbol{u}}
\def\H{\mathbb{H}}

\newcommand{\R}{\mathbb{R}}

\renewcommand{\d}{\/\mathrm{d}\/}

\newcommand{\Addresses}{{
		\footnote{
			
			\noindent \textsuperscript{1,2}Department of Mathematics, Indian Institute of Technology Roorkee-IIT Roorkee,
			Haridwar Highway, Roorkee, Uttarakhand 247667, INDIA.\par\nopagebreak
			\noindent 	\textit{e-mail:} \texttt{Pardeep Kumar: pkumar3@ma.iitr.ac.in.}
			
			\noindent  \textit{e-mail:} \texttt{Manil T. Mohan: manilfma@iitr.ac.in, maniltmohan@gmail.com.}
			
			\noindent \textsuperscript{*}Corresponding author.

			\textit{Key words:} Convective Brinkman-Forchheimer equations, Locally distributed control, Carleman estimates, Null controllability,  Local exact controllability.
			
			Mathematics Subject Classification (2020): Primary 35Q35; Secondary 93B05, 93B07, 93C10, 93C20.

}}}

\begin{document}
	
	
	\title[Local exact controllability of the CBF equations]{Local exact controllability to the trajectories of the convective Brinkman-Forchheimer equations
		\Addresses}
	\author[P. Kumar and M. T. Mohan ]{Pardeep Kumar\textsuperscript{1} and Manil T. Mohan\textsuperscript{2*}}

	\maketitle
	
	\begin{abstract}
		In this article, we discuss the local exact controllability to trajectories of the following  convective Brinkman-Forchheimer (CBF) equations (or damped Navier-Stokes equations) defined	in a bounded domain $\Omega
		\subset\mathbb{R}^d$ ($d=2,3$) with smooth boundary:
		\begin{align*}
			\frac{\partial\u}{\partial t}-\mu \Delta\boldsymbol{u}+(\boldsymbol{u}\cdot\nabla)\boldsymbol{u}+\alpha\boldsymbol{u}+\beta|\boldsymbol{u}|^{2}\boldsymbol{u}+\nabla p=\boldsymbol{f}+\boldsymbol{\vartheta}, \ \ \  \nabla\cdot\boldsymbol{u}=0,
		\end{align*}
		where the control $\boldsymbol{\vartheta}$ is distributed in a subdomain $\omega \subset \Omega$, and the parameters $\alpha,\beta,\mu>0$ are constants. We first present global Carleman estimates and observability inequality for the  adjoint problem of a linearized version of CBF equations by using a global Carleman estimate for the Stokes system. This allows us to obtain its null controllability at any time $T>0$. We then use the inverse mapping theorem to deduce local results concerning the exact controllability to the trajectories of CBF equations.
	\end{abstract}
	\section{Introduction}\label{sec1}\setcounter{equation}{0}

	Let $\Omega \subset \R^d$ $(d=2,3)$ be a bounded domain with a smooth boundary $\partial \Omega$ (for example of class $\C^2$). Let $\omega \subset \Omega$ be a non-empty, small, open subset, and let $T>0$. We will use the notations $Q:=\Omega\times(0,T)$, $Q_{\omega}:=\omega\times(0,T)$, and $\Sigma:=\partial\Omega\times(0,T)$, and denote by $\boldsymbol{n}(x),$ the outward unit normal to $\Omega$ at the point $x\in \partial{\Omega}$. The convective Brinkman-Forchheimer (CBF) equations describe the motion of an incompressible viscous fluid through a rigid, homogeneous, isotropic, porous medium (cf. \cite{AB}). Let us consider the following controlled  CBF equations: 
	\begin{equation}\label{a1}
		\left\{
		\begin{aligned}
			\frac{\partial\u}{\partial t}-\mu \Delta\u+(\u\cdot\nabla)\u+\alpha\u+\beta|\u|^{2}\u+\nabla p&=\f+\boldsymbol{\vartheta} \mathds{1}_{\omega}, \ \text{ in }  Q, \\ \nabla\cdot\u&=0, \ \text{ in }  Q, \\ 
			\u&=\mathbf{0}, \ \text{ on }  \Sigma, \\
			\u(0)&=\u_0,  \text{ in }  \Omega.
		\end{aligned}
		\right.
	\end{equation}
	Here $\u(x,t) \in \R^d$ represents the velocity vector field of the fluid, $p(x,t)\in\R$ is the pressure in the fluid,  $\boldsymbol{f}(x,t)\in\R^d$ is a given  external forcing acting on the fluid, $\boldsymbol{\vartheta}(x,t) \in \R^d$ is a control distributed in some arbitrary fixed subdomain $\omega$ of the physical domain $\Omega$, $\u_0(x)$ is a given initial vector field, and $\mathds{1}_{\omega}$ is the characteristics function of the set $\omega$:
	\begin{align*}
		\mathds{1}_{\omega}(x)= \left\{\begin{array}{cc}1,&\text{for } x \in \omega,\\
			0,& \ \ \text{ for } x \in \Omega \backslash \omega.\end{array}\right.
	\end{align*}
	The positive constant $\mu$ denotes the \emph{Brinkman coefficient} (effective viscosity), while  $\alpha$ and $\beta$ stand for  the \emph{Darcy} (permeability of porous medium) and \emph{Forchheimer} (proportional to the porosity of the material) coefficients, respectively. Setting $\alpha=\beta=0$ yields the classical $d$-dimensional controlled Navier-Stokes equations  (NSE). This accounts to refer the problem \eqref{a1} as a controlled damped NSE also. The global solvability results of the uncontrolled problem \eqref{a1} (that is, $\boldsymbol{\vartheta}=\mathbf{0}$)  is discussed in \cite{CLF,MTM}, etc.

	\subsection{Statement of the problem} 
	The main objective of this work is to discuss the local exact controllability to trajectories of the CBF equations \eqref{a1}.  	To state the problem precisely and present the main result, we require some function spaces which are given below (\hspace {-0.2mm}\cite{RT}):
	\begin{align*}
		\H&=\{\u\in \mathrm{L}^2(\Omega)^d: \ \nabla\cdot\u=0  \ \ \text{in}  \ \Omega  \ \ \text{and} \ \ \u\cdot\boldsymbol{n}=0  \ \ \text{on} \ {\partial\Omega}\},\\
		\V&=\{\u\in\mathrm{H}_0^1(\Omega)^d: \ \nabla\cdot\u=0 \ \ \text{in}  \ \Omega\},\\
		\mathrm{L}^2_0(\Omega)&=\bigg\{q \in\mathrm{L}^2(\Omega): \ \int_{\Omega} q(x) \d x=0 \bigg\},
	\end{align*}
where $\boldsymbol{n}$ is the unit outward drawn normal to the boundary and $\u\cdot\boldsymbol{n}\big|_{\partial\Omega}=0 $ should be understood in the sense of trace (\hspace {-0.2mm}\cite[Section 1.3, Chapter 1]{RT}). The following notations are also used repeatedly in the article:  $\mathrm{L}^p(Q)^d:=\mathrm{L}^p(0,T;\mathrm{L}^p(\Omega)^d)$, for $p \in [1,\infty]$,  $\mathrm{L}^2(Q_{\omega})^d:=\mathrm{L}^2(0,T;\mathrm{L}^2(\omega)^d)$, $\mathrm{L}^2(Q)^{d^2}:=\mathrm{L}^2(0,T;\mathrm{L}^2(\Omega)^{d^2})$, $\mathrm{L}^2_0(Q):=\mathrm{L}^2(0,T;\mathrm{L}^2_0(\Omega))$ and $\mathrm{L}^2(Q):=\mathrm{L}^2(0,T;\mathrm{L}^2(\Omega))$. 	Let $(\cdot,\cdot)$ denote the inner product in the Hilbert space $\H$ and $\langle\cdot,\cdot\rangle$ represent the duality pairing between a Banach space $\mathbb{X}$ and its dual $\mathbb{X}'$.  Note that $\H$ can be identified with its own dual $\H'$. We will denote a generic constant (usually depending on $\Omega$ and $\omega$) by $C$ throughout this article.

	We cannot expect the exact controllability for  CBF equations with an arbitrary target function due to the dissipative and irreversible characteristics of the system (\hspace {-0.2mm}\cite{RTR}). Therefore, we consider an ``ideal" trajectory $(\wi{\u}, \wi{p})$ of the following uncontrolled CBF equations:
	\begin{equation}\label{a2}
		\left\{
		\begin{aligned}
			\frac{\partial \wi{\u}}{\partial t}-\mu \Delta	\wi{\u}+(	\wi{\u}\cdot\nabla)	\wi{\u}+\alpha	\wi{\u}+\beta|	\wi{\u}|^{2}	\wi{\u}+\nabla 	\wi{p}&=\f, \ \text{ in } Q, \\ \nabla\cdot	\wi{\u}&=0, \  \text{ in }  Q, \\
			\wi{\u}&=\mathbf{0}, \  \text{ on }  \Sigma, \\ 
			\wi{\u}(0)&=\wi{\u}_0,  \text{ in }  \Omega,
		\end{aligned}
		\right.
	\end{equation}
	with the same right-hand side  $\f$  as in \eqref{a1}. We  additionally suppose that $(\wi{\u},\wi{p})$ satisfies the following regularity property:
	\begin{align}\label{a2a}
		\wi{\u} \in \mathrm{L}^\infty(Q)^d.
	\end{align}
	
	Let us now define the notion of local exact controllability to trajectories of the CBF equations \eqref{a1}. The task is to find a control $\boldsymbol{\vartheta}$ such that at least one solution of \eqref{a1} verifies that
	\begin{align*}
		\u(\cdot,T)=\wi{\u}(\cdot,T), \ \text{ in }  \Omega.
	\end{align*}
	If we are able to obtain such a control, we can switch it off  after time $T$ and the system will follow the ``ideal" trajectory $(\wi{\u},\wi{p})$ of \eqref{a2}.
	
	More precisely, we say that the system \eqref{a1} is \emph{locally exactly controllable to  trajectories} if for  a suitable trajectory $(\wi{\u},\wi{p}),$ solution of \eqref{a2},  there  exists  a $\delta >0$ such that  if 
	\begin{align*}
		\|\u_0-\wi{\u}_0\|_{\mathrm{L}^{2d-2}(\Omega)^d} \leq \delta,
	\end{align*}
	we can find a control $\boldsymbol{\vartheta}$ such that the corresponding solution to \eqref{a1} satisfies:
	\begin{align}\label{a2b}
		\u(\cdot,T)=\wi{\u}(\cdot,T), \  \text { in }  \Omega.
	\end{align}
	\subsection{Main result and application}
	The main result of the present article concerns the local exact controllability to trajectories of the CBF equations \eqref{a1}. We now state our main result in the following theorem:
	\begin{theorem}\label{thmm}
		Let us assume that $\omega$ is a non-empty open subset of $\Omega$ and $T>0$. We suppose that the solution $(\wi{\u},\wi{p})$ of the  system \eqref{a2} satisfies \eqref{a2a}. Then, there exists $\delta >0$ such that for any $\u_0 \in \H \cap \mathrm{L}^4(\Omega)^d$ satisfying  \begin{align*}
			\| \u_0-\wi{\u}_0\|_{\mathrm{L}^{2d-2}(\Omega)^d} \leq \delta,
		\end{align*} 
		there exists a control $\boldsymbol{\vartheta} \in \mathrm{L}^2(Q_\omega)^d$ and a solution $(\u,p)$ of the system  \eqref{a1} that satisfies \eqref{a2b}.
	\end{theorem}
Motivated 	by  the ideas developed in  \cite{CGIP,Im1,Puel}, we prove  Theorem \ref{thmm} in  Section \ref{sec5} and the proof is based on an application of the inverse mapping theorem.
	\begin{remark}
	Under the	regularity hypotheses described in \eqref{a2a}, J.-P. Puel \cite{Puel} established local exact controllability to the trajectories of the Navier-Stokes system. Therefore, it will be proven in this article that the presence of a cubic nonlinearity requires no additional regularity on the trajectories than the one described in \eqref{a2a}.
	\end{remark}
	Let us provide an application of the above controllability result (for more details, see \cite{AVIm,Puel}, etc).
	\vskip 0.2cm
	\noindent
	\textbf{Stabilizability of unstable stationary flows:} Let $({\u}^\dag,{p}^\dag)$ be a stationary solution of the following CBF equations in a bounded domain $\Omega \subset \R^d$:
	\begin{equation}\label{a21}
		\left\{
		\begin{aligned}
			&-\mu \Delta	{\u}^\dag+(	{\u}^\dag \cdot\nabla)	{\u}^\dag+\alpha	{\u}^\dag+\beta|{\u}^\dag|^{2}	{\u}^\dag+\nabla 	{p}^\dag=\f(x), \ \text{ in } \Omega, \\ &\nabla\cdot	{\u}^\dag=0, \  \text{ in }  \Omega,  \ \ \ \
			{\u}^\dag=\mathbf{0}, \  \text{ on }  \partial \Omega.
		\end{aligned}
		\right.
	\end{equation}
The global solvability results (${\u}^\dag\in\D(\A)$, where $\A$ is the Stokes operator) of the system \eqref{a21} can be obtained from \cite{MTM1}. 	Suppose that this stationary solution is unstable. 
Thus, for each $\delta >0,$ there is an initial condition $\u_0 \in \{\u :\|\u-{\u}^\dag\|_{\mathrm{L}^{2d-2}(\Omega)^d} < \delta\}$ such that the solution $(\u,p)$ of the CBF equations \eqref{a1} with $\f(x,t) \equiv \f(x)$ and $\boldsymbol{\vartheta}(x,t)\equiv \mathbf{0}$  would diverge from $({\u}^\dag, {p}^\dag)$ as $t \rightarrow \infty$ and 
	\begin{align*}
		\|\u(\cdot,t)-{\u}^\dag\|_{\mathrm{L}^4(\Omega)^d} \nrightarrow 0 \ \ \text{as} \ \ t \rightarrow \infty.
	\end{align*}
	One can stabilize the unstable stationary solution  $({\u}^\dag,{p}^\dag)$ by replacing the first equation of \eqref{a1} by the distributed control $\boldsymbol{\vartheta}$. Indeed, by Theorem \ref{thmm}, for any initial
	condition $\u_0 \in \H \cap\mathrm{L}^{2d-2}(\Omega)^d,$ there is a  control $\boldsymbol{\vartheta}$ supported in  $\omega \times (0,T)$, such that the solution  $(\u,p)$ of the CBF equations \eqref{a1} satisfies the  condition 
	$$ \u(x,T)\equiv {\u}^\dag(x).$$
	This says that we are able to perform a strong stabilization of the unstable stationary flows $({\u}^\dag, {p}^\dag)$ of solution of the system \eqref{a21} by an interior control.
	\subsection{Strategy and literature review}
	Several steps will be required to prove the main  Theorem \ref{thmm}. Let us now sketch the strategy which we  employ to solve our problem.
	\begin{itemize}
		\item
		First of all if we write
		$$\v=\u-\wi{\u}, \ \ \ q=p-\wi{p}, \ \ \v_0=\u_0-\wi{\u}_0,$$ we obtain a new controlled system:
		\begin{equation}\label{a3}
			\left\{
			\begin{aligned}
				\mathcal{L}\v+\nabla q&=\mathcal{N} (\v)+\boldsymbol{\vartheta} \mathds{1}_{\omega}, \ \text{ in } Q, \\ \nabla\cdot\v&=0, \ \text{ in }  Q,\\
				\v&=\mathbf{0}, \ \text{ on } \Sigma, \\ 	\v(0)&=\v_0,  \text{ in }  \Omega,
			\end{aligned}
			\right.
		\end{equation}
		where
		\begin{align}
			\mathcal{L}\v &= \frac{\partial\v}{\partial t}-\mu \Delta\v+(\wi{\u} \cdot \nabla )\v + (\v \cdot \nabla)\wi{\u} +\alpha\v+\beta|\wi{\u}|^{2}\v+2 \beta (\v \cdot \wi{\u}) \wi{\u}, \label{a31} \\
			\mathcal{N} (\v)&=-(\v \cdot \nabla)\v-\beta |\v|^2(\v+\wi{\u})-2\beta (\v \cdot \wi{\u}) \v.\label{a32}
		\end{align}
		Consequently, it is easy to see that the local exact controllability of the trajectories of the system \eqref{a1} is equivalent to the local null controllability of the system \eqref{a3}, that is, we want to find the control $\boldsymbol{\vartheta}$ such that at time $T$, we have 
	$
			\v(\cdot,T)=\mathbf{0}.
	$
		More precisely, we say that the system \eqref{a3} is \emph{null controllable} if, for any initial data $\v_0 \in \H \cap \mathrm{L}^{2d-2}(\Omega)^d$, there exists a control $\boldsymbol{\vartheta}$ such that the associated solution to the system \eqref{a3} satisfies:
		\begin{align}\label{a4}
			\v(\cdot,T)=\mathbf{0}, \ \ \ \text{in } \Omega.
		\end{align}
		Accordingly, we will focus on proving the local null controllability of the nonlinear system \eqref{a3}.
		\item In order to prove this local result, we employ an inverse mapping argument, developed in \cite{CGIP,Puel}, etc.,  in two steps.
		
		Given an appropriate choice of tensor $\g$, a function $\h$ (in a functional class that will be specified later), and given an initial data $\v_0 \in \H$, we prove the null controllability of the following linearized system:
		\begin{equation}\label{a5}
			\left\{
			\begin{aligned}
				\mathcal{L}\v +\nabla q &=\nabla \cdot \g+\h+\boldsymbol{\vartheta} \mathds{1}_{\omega}, \ \text{ in } Q,  \\ \nabla\cdot\v&=0, \ \text{ in }  Q, \\
				\v&=\mathbf{0},  \ \text{ on } \Sigma, \\
				\v(0)&=\v_0,  \text{ in }  \Omega,
			\end{aligned}
			\right.
		\end{equation}
		where the operator $\mathcal{L}$ is defined in \eqref{a31}.
		
		The null controllability of \eqref{a5} is based on the attainment of the so-called \emph{observability inequality} for a backward system associated with \eqref{a5}. Let us introduce the adjoint state $(\boldsymbol{\varphi}, \pi)$ associated with \eqref{a5} that satisfies the system:
		\begin{equation}\label{a12}
			\left\{
			\begin{aligned}
				\mathcal{L}^{*} \vr+\nabla \pi&=\mathbf{0} , \ \text{ in }  Q,\\   \nabla\cdot\vr&=0,  \ \text{ in } Q, 
				\\ \vr&=\mathbf{0},  \ \text{ on }  \Sigma, \\
				\vr(T)&=\frac{1}{\varepsilon}\v_{\varepsilon}(T),  \ \text{ in }  \Omega.
			\end{aligned}
			\right.
		\end{equation}
		Here $\mathcal{L}^*$ is the adjoint operator of $\mathcal{L}$ define by
		\begin{align}\label{a121}
			\mathcal{L}^*\vr = -\frac{\partial\vr}{\partial t}-\mu \Delta\vr-(D \vr) \wi{\u}+\alpha\vr+\beta|\wi{\u}|^{2}\vr+2 \beta (\vr \cdot \wi{\u}) \wi{\u},
		\end{align}
 $\v_\varepsilon(\cdot,\boldsymbol{\vartheta}_{\varepsilon})$ is the solution of the optimal control problem (cf. subsection \ref{subsec4.1})
		\begin{align*}
			\min_{\boldsymbol{\vartheta} \in \mathrm{L}^2(Q_{\omega})^d} \bigg(\frac{1}{2\varepsilon}\|\v(T)\|_{\H}^2+\frac{1}{2}\int_{Q_{\omega}}|\boldsymbol{\vartheta}|^2\d x\d t\bigg),
		\end{align*}
	such that $\v(\cdot,\boldsymbol{\vartheta})$ is the solution of  the system \eqref{a5} associated to $\boldsymbol{\vartheta}$, and $D \vr$ stands for the symmetrized gradient $D \vr=\nabla \vr+(\nabla \vr)^{\top}$. 
		
		At present, the most powerful tool to prove the null controllability of the linearized system \eqref{a5} is establishing global Carleman estimates for the solutions of its adjoint system \eqref{a12}. We first deduce a Carleman estimate for the adjoint system \eqref{a12} and with the help of this Carleman estimate, we derive the observability inequality (see Theorem \ref{thmo}) of \eqref{a12}, which is useful for proving the null controllability of the system \eqref{a5}.
		
		\item Next, we  return to the nonlinear problem and then apply the inverse mapping theorem to extend the control result from the  linearized system \eqref{a5} to the nonlinear system \eqref{a3}, thus obtaining the desired result (proof of main Theorem \ref{thmm}).
	\end{itemize}
	
	This two-step strategy has been successfully applied to other nonlinear systems such as the incompressible Navier-Stokes equations (see \cite{CGIP,FIm,Im3,Im1,Puel}, etc.),  the Boussinesq system \cite{AVIm,SG}, etc., the one-dimensional Kuramoto-Sivashinsky equation \cite{EC}, the Swift-Hohenberg equation \cite{PGAO}, and the Cahn-Hilliard equation \cite{PG}, etc. and  references therein.
	
	The controllability problems for incompressible fluids have been the subject of much research in recent years.  There are several results concerning the controllability properties of the  incompressible NSE by Coron \cite{JMC,JMC1}, Coron and Fursikov  \cite{JMC2}, Fursikov \cite{Avf}, Fursikov and Imanuilov  \cite{FYIm,FYIm1,FIm, FIm 1,AVIm}, Imanuilov  \cite{Im2,Im21,Im3,Im1}, Fernández-Cara \cite{FC}, Fernández-Cara et. al.  \cite{CGIP,EFC1}, Guerrero \cite{SG1},  etc. and references therein. Several papers have been investigated on the controllability properties of the Boussinesq system by Fursikov and Imanuilov \cite{FIm 2,AVIm}, Imanuilov  \cite{Im21}, Fernández-Cara et. al. \cite{EFC1}, Guerrero \cite{SG}, Burgos, Guerrero and Puel \cite{MGB},  etc. and references therein.
	The above works explore the null controllability, approximate controllability, exact controllability to the trajectories using distributed or boundary control, and also controlling the system with a reduced number of controls.
	
	In 1998, O. Imanuvilov established a result of local exact controllability to trajectories for NSE with classical Dirichlet boundary conditions, under rather strong regularity assumptions on the trajectory in \cite{Im3}. This study was improved in \cite{Im1} and furthermore, this result was extended by Fernández-Cara et. al. \cite{CGIP} and Puel  \cite{Puel} with less regularity on the trajectory.
%
	
	Guerrero  \cite{SG1} considered the local exact controllability to the trajectories of NSE with nonlinear Navier-slip boundary conditions, with the control being supported in a small set. The local exact controllability to the trajectories of the  Boussinesq system is studied in \cite{SG}. The local exact controllability to the trajectories for the $N$-dimensional Navier-Stokes and Boussinesq systems with $N-1$ scalar controls is demonstrated by Fernández-Cara et. al.  \cite{EFC1}. They also obtained  global null controllability results for some (truncated) approximations of NSE. The authors  in \cite{MGB} proved  local exact controllability to the trajectories of  Boussinesq systems when the control is supported in a small set and acting on the divergence equation. For an extensive study on the controllability problems for NSE and related models, the interested readers are referred to see \cite{BES,MGB,CGIP,SG,SG1,Im1,Puel}, etc. and references therein. The aim of this article is to prove local exact controllability to the trajectories of the CBF equations \eqref{a1}. 
	To the best of our knowledge, there are no existing results in the literature on the exact controllability of CBF equations, and this work appears to be the first one to discuss local exact controllability to the trajectories of the CBF equations \eqref{a1}.

	\subsection{Organization of the article}
	The rest of the article is structured as follows: In the next section, we  provide the well-posedness result (Theorem \ref{thmwell}) for the linearized CBF equations \eqref{a5}. In section \ref{sec3}, we first discuss the necessary weight functions needed to obtain the main results of this article. In the same section, we  derive the required Carleman estimates (see \eqref{o23}), and observability inequality (Theorem \ref{thmo}) for the adjoint system \eqref{a12} associated to the linearized system \eqref{a5} by using the global Carleman estimates of the Stokes system. In Section \ref{sec4}, we prove the null controllability of the linearized system \eqref{a5} (Theorem \ref{thmn1}). Furthermore, we also show that the solutions and controls decay exponentially (Theorem \ref{thmE}) which is helpful to examine the local exact controllability to the trajectories of the CBF equations \eqref{a1}. In the final section, firstly, we define some functional classes and deduce some useful regularity results that are required to handle the nonlinear terms (see subsection \ref{sub5.2}). Finally, we establish the local null controllability of the nonlinear problem  \eqref{a3}  by an application of the inverse mapping theorem which leads to the required results.

	\section{Well-posedness of the Linearized CBF Equations}\label{sec2}\setcounter{equation}{0}
	In this section, we provide the well-posedness result for the linearized CBF equations \eqref{a5}, which is crucial in obtaining the controllability result for the CBF equations.
	\subsection{Well-posedness} Let us prove the existence, uniqueness, and stability results for  the linearized system \eqref{a5}. To do this, we first derive formal  energy estimates for the solution of the problem \eqref{a5}. The following theorem establishes the well-posedness of  the system \eqref{a5}.
	\begin{theorem}\label{thmwell}
		Let  $\wi{\u} \in \mathrm{L}^\infty(Q)^d$, $\v_0 \in \H$, $\g \in \mathrm{L}^2(Q)^{d^2}$,  $\h \in \mathrm{L}^2(0,T;\mathrm{H}^{-1}(\Omega)^d)$, and $\boldsymbol{\vartheta} \in \mathrm{L}^2(Q_{\omega})^d$ be given. Then, there exists a unique weak solution $(\v,q)$ to the system \eqref{a5} satisfying  
		\begin{align*}
			\v \in \C([0,T];\H) \cap \mathrm{L}^2(0,T;\V), \ \ \ \text{and} \ \ \ q \in \mathrm{L}^2_0(Q),
		\end{align*}
		and there exists a constant $C>0$ such that
		\begin{align}\label{a7}
			&	\|\v\|_{\C([0,T];\H)} + \|\nabla\v\|_{\mathrm{L}^2(0,T;\H)} +\| q\|_{\mathrm{L}^2(Q)} \no\\&\leq C\big(\|\v_0\|_{\H}+\|\g \|_{ \mathrm{L}^2(Q)^{d^2}}+\|\h \|_{ \mathrm{L}^2(0,T;\mathrm{H}^{-1}(\Omega)^d)}+\|\boldsymbol{\vartheta}\|_{\mathrm{L}^2(Q_{\omega})^d}\big),
		\end{align}
		and
		\begin{align}\label{a71}
			&\|\v_1-\v_2\|_{\C([0,T];\H)}+\|\nabla(\v_1-\v_2)\|_{\mathrm{L}^2(0,T;\H)}+\| q_1-q_2\|_{\mathrm{L}^2(Q)} \no\\& \leq
			C	\bigg\{\|\v_1(0)-\v_2(0)\|_{\H} + \|\g_1-\g_2\|_{\mathrm{L}^2(Q)^{d^2}}  + \|\h_1-\h_2\|_{\mathrm{L}^2(0,T;\mathrm{H}^{-1}(\Omega)^d)} +
			\|\boldsymbol{\vartheta}_1-\boldsymbol{\vartheta}_2\|_{\mathrm{L}^2(Q_{\omega})^{d}}  \bigg\},
		\end{align}
		where $C$ depends on the input data,  $\mu,\alpha,\beta,T$, and $\Omega$.
	\end{theorem}
	\begin{proof}
		\textbf{Existence:} The following calculations can be made rigorous by using a Faedo-Galerkin approximation technique. As the system \eqref{a5}  is linear, weak convergence of a Faedo-Galerkin approximated subsequence  is enough to pass to the limit. 	Taking the inner product with $\v(\cdot)$ to the first equation in \eqref{a5} and using the fact that $((\wi{\u} \cdot \nabla )\v,\v)=0$ and $(\nabla q, \v)=0$, we obtain
		\begin{align}\label{a6}
			&	\frac{1}{2}\frac{\d}{\d t}\|\v(t)\|_{\H}^2+\mu \|\nabla \v(t)\|_{\H}^2+\alpha \|\v(t)\|_{\H}^2 + \beta\||\wi{\u}||\v|\|_{\H}^2+2\beta\|\wi{\u} \cdot\v\|_{\H}^2\nonumber\\&\leq (\nabla \cdot \g,\v)+\langle\h,\v\rangle+(\boldsymbol{\vartheta}\mathds{1}_{\omega},\v)-((\v \cdot \nabla)\wi{\u},\v),
		\end{align}
		for a.e. $t\in[0,T]$, where we have performed the integration by parts over $\Omega$. We estimate each term on the right-hand side of \eqref{a6} separately by employing H\"older's and Young's inequalities as follows:
		\begin{align*}
			|(\nabla \cdot \g,\v)| &\leq \|\g\|_{\mathrm{L}^2(\Omega)^{d^2}} \|\nabla \v\|_{\H} \leq \frac{3}{2\mu}\|\g\|_{\mathrm{L}^2(\Omega)^{d^2}}^2 +\frac{\mu}{6}\|\nabla\v\|_{\H}^2, \\
			|\langle\h,\v\rangle| &\leq \|\h\|_{\mathrm{H}^{-1}(\Omega)^d} \|\v\|_{\V} \leq \frac{3}{2 \mu}\|\h\|_{\mathrm{H}^{-1}(\Omega)^d}^2 +\frac{\mu}{6}\|\v\|_{\V}^2, \\
			|(\boldsymbol{\vartheta}\mathds{1}_{\omega},\v)|&\leq \|\boldsymbol{\vartheta}\|_{\mathrm{L}^2(\omega)^{d}} \|\v\|_{\H} \leq \frac{1}{2\alpha}\|\boldsymbol{\vartheta}\|_{\mathrm{L}^2(\omega)^{d}}^2 +\frac{\alpha}{2}\|\v\|_{\H}^2, \\
			|((\v \cdot \nabla)\wi{\u},\v)| &=|((\v \cdot \nabla)\v,\wi{\u})| \leq \|\nabla\v\|_{\H}\|\v\|_{\mathrm{L}^{4}(\Omega)^{d}}\|\wi{\u}\|_{\mathrm{L}^{4}(\Omega)^{d}} \\& \leq C \|\nabla\v\|_{\H}^{1+\frac{d}{4}}\|\v\|_{\H}^{1-\frac{d}{4}}\|\wi{\u}\|_{\mathrm{L}^{4}(\Omega)^{d}} \leq C\|\wi{\u}\|_{\mathrm{L}^{4}(\Omega)^{d}}^{\frac{8}{4-d}}\|\v\|_{\H}^2+\frac{\mu}{6}\|\nabla\v\|_{\H}^2.
		\end{align*} 
		Substituting the above estimates  in \eqref{a6},  and integrating the resulting estimate from $0$ to $t$,   we obtain
		\begin{align}\label{a61}
			&\|\v(t)\|_{\H}^2+\mu \int_0^t\|\nabla\v(s)\|_{\H}^2\d  s +2 \beta\int_0^t\||\wi{\u}(s)||\v(s)|\|_{\H}^2\d s +4\beta\int_0^t\|\wi{\u}(s) \cdot\v(s)\|_{\H}^2\d s\no\\& \leq\|\v_0\|_{\H}^2 + \frac{3}{\mu}\int_0^t\|\g(s)\|_{\mathrm{L}^2(\Omega)^{d^2}}^2 \d s+ \frac{3}{\mu}\int_0^t\|\h(s)\|_{\mathrm{H}^{-1}(\Omega)^{d}}^2 \d s \no\\& \quad+
			\frac{1}{\alpha}\int_0^t\|\boldsymbol{\vartheta}(s)\|_{\mathrm{L}^2(\omega)^{d}}^2 \d s +C \int_0^t\|\wi{\u}(s)\|_{\mathrm{L}^{4}(\Omega)^{d}}^{\frac{8}{4-d}}\|\v(s)\|_{\H}^2 \d s,
		\end{align} 
		for all $t \in [0,T]$. An application of Gronwall's inequality in \eqref{a61} yields
		\begin{align*}
			\|\v(t)\|_{\H}^2   & \leq \bigg(
			\|\v_0\|_{\H}^2 + \frac{3}{\mu}\int_0^T\|\g(t)\|_{\mathrm{L}^2(\Omega)^{d^2}}^2 \d t+ \frac{3}{\mu}\int_0^T\|\h(t)\|_{\mathrm{H}^{-1}(\Omega)^{d}}^2 \d t 	\no\\& \quad+
			\frac{1}{\alpha}\int_0^T\|\boldsymbol{\vartheta}(t)\|_{\mathrm{L}^2(\omega)^{d}}^2 \d t \bigg)
			\times \exp\bigg\{C \int_0^T\|\wi{\u}(t)\|_{\mathrm{L}^{4}(\Omega)^{d}}^{\frac{8}{4-d}} \d t \bigg\},
		\end{align*}
		for all $t \in [0,T]$. Thus, from \eqref{a61}, it is immediate that
		\begin{align}\label{a62}
			&\sup_{t\in[0,T]}	\|\v(t)\|_{\H}^2 +\int_0^T\|\nabla\v(t)\|_{\H}^2\d t+\int_0^T\|\wi{\u}(t) \cdot\v(t)\|_{\H}^2\d t+\int_0^T\||\wi{\u}(t)||\v(t)|\|_{\H}^2\d t
			\nonumber\\& \leq C\bigg(
			\|\v_0\|_{\H}^2 + \|\g\|_{\mathrm{L}^2(Q)^{d^2}}^2 + \|\h\|_{\mathrm{L}^2(0,T;\mathrm{H}^{-1}(\Omega)^{d})}^2 +
			\|\boldsymbol{\vartheta}\|_{\mathrm{L}^2(Q_{\omega})^{d}}^2  \bigg).
		\end{align}
		For $\Phi \in \mathrm{L}^2(0,T;\V)$, from the first equation in \eqref{a5}, we obtain
		\begin{align*}
			|\langle\v_t,\Phi\rangle|&= |\langle \nabla \cdot \g+\h+\boldsymbol{\vartheta}\mathds{1}_{\omega}+\mu \Delta \v-(\v \cdot \nabla)\wi{\u}-(\wi{\u} \cdot \nabla) \v\\& \quad-\alpha \v -\beta|\wi{\u}|^{2}\v-2 \beta (\v \cdot \wi{\u}) \wi{\u}, \Phi \rangle| 
			\\& \leq \big(\|\nabla \cdot \g\|_{\mathrm{H}^{-1}(\Omega)^{d}}+\|\h\|_{\mathrm{H}^{-1}(\Omega)^{d}}+\|\boldsymbol{\vartheta}\mathds{1}_{\omega}\|_{\mathrm{H}^{-1}(\Omega)^{d}}+\mu \|\Delta \v \|_{\V'}+\alpha \|\v\|_{\V'} \\&\quad+3\beta \||\wi{\u}|^{2}|\v | \|_{\V'} \big) \|\Phi\|_{\V}+| \langle (\v \cdot \nabla)\Phi,\wi{\u} \rangle| +| \langle (\wi{\u} \cdot \nabla)\Phi,{\v} \rangle| 
			\\& \leq C\big(\|\nabla \cdot \g\|_{\mathrm{H}^{-1}(\Omega)^{d}}+\|\h\|_{\mathrm{H}^{-1}(\Omega)^{d}}+\|\boldsymbol{\vartheta}\mathds{1}_{\omega}\|_{\mathrm{H}^{-1}(\Omega)^{d}}+\| \v \|_{\V}\\&\quad+ \|\v\|_{\H} + \||\wi{\u}|^{2}|\v | \|_{\mathrm{L}^{\frac{4}{3}}(\Omega)^{d}} \big) \|\Phi\|_{\V}+C\|\wi{\u}\|_{\mathrm{L}^{4}(\Omega)^{d}} \|\v\|_{\mathrm{L}^{4}(\Omega)^{d}}\|\Phi\|_{\V}
			\\& 
			\leq  C\big\{\| \g\|_{\mathrm{L}^2(\Omega)^{d^2}}+\|\h\|_{\mathrm{H}^{-1}(\Omega)^{d}}+\|\boldsymbol{\vartheta}\|_{\mathrm{L}^2(\omega)^{d}}+\|\nabla \v\|_{\H}\\&\quad+\||\wi{\u}||\v |\|_{\H}\|\wi{\u}\|_{\mathrm{L}^{4}(\Omega)^d}+\|\wi{\u}\|_{\mathrm{L}^{4}(\Omega)^d}\|\nabla \v\|_{\H} \big\} \|\Phi\|_{\V},
		\end{align*}
		where we have used H\"older's  inequality and embedding $ \V \hookrightarrow \mathrm{L}^{4}_{\sigma}(\Omega)^d  \hookrightarrow \H \equiv \H' \hookrightarrow \mathrm{L}_{\sigma}^{\frac{4}{3}}(\Omega)^d  \hookrightarrow \V'$. Integrating the above estimate from $0$ to $T$, and using the energy estimate \eqref{a62}, along with the hypothesis $\wi{\u} \in \mathrm{L}^\infty(Q)^d$,  $\v_0 \in \H$,  $\g \in \mathrm{L}^2(Q)^{d^2}$, $\boldsymbol{\vartheta} \in \mathrm{L}^2(Q_{\omega})^d$, and $\h \in \mathrm{L}^2(0,T;\mathrm{H}^{-1}(\Omega)^{d})$,  we deduce that $\v_t \in \mathrm{L}^2(0,T;\V')$, and  we also  have
		\begin{align}\label{a621}
			\|\v_t\|_{\mathrm{L}^2(0,T;\V')}^2 \leq C\bigg(
			\|\v_0\|_{\H}^2 + \|\g\|_{\mathrm{L}^2(Q)^{d^2}}^2 + \|\h\|_{\mathrm{L}^2(0,T;\mathrm{H}^{-1}(\Omega)^{d})}^2 +
			\|\boldsymbol{\vartheta}\|_{\mathrm{L}^2(Q_{\omega})^{d}}^2  \bigg).
		\end{align}
		The fact that
		\begin{align*}
			\v \in \mathrm{L}^2(0,T;\V) \ \ \ \text{and} \ \ \ \v_t \in \mathrm{L}^2(0,T;\V'),
		\end{align*}
		implies that $\v \in \C([0,T];\H)$. 
		
		The first equation in the system \eqref{a5} can be used directly to deduce the pressure field $q$ by $\v_0,\v, \g, \h, \boldsymbol{\vartheta}$. Taking divergence on both sides of the first equation in \eqref{a5}, we get
		\begin{align*}
			\Delta q = \nabla \cdot \big\{\nabla \cdot \g+\h+\boldsymbol{\vartheta}\mathds{1}_{\omega}-(\wi{\u} \cdot \nabla)\v-(\v \cdot \nabla)\wi{\u} -\beta|\wi{\u}|^{2}\v-2 \beta (\v \cdot \wi{\u}) \wi{\u}\big\},
		\end{align*}
		in the weak sense. From the above equation, we have
		\begin{align}\label{a622}
			q=(\Delta)^{-1}\big[\nabla \cdot \big\{\nabla \cdot \g+\h+\boldsymbol{\vartheta}\mathds{1}_{\omega}-(\wi{\u} \cdot \nabla)\v-(\v \cdot \nabla)\wi{\u} -\beta|\wi{\u}|^{2}\v-2 \beta (\v \cdot \wi{\u}) \wi{\u}\big\} \big].
		\end{align}
		Taking $\mathrm{L}^2$-norm on both sides of the equation \eqref{a622}, and then using elliptic regularity (Cattabriga's regularity theorem) and H\"older's inequality, we obtain
		\begin{align*}
			\|q\|_{\mathrm{L}^2(\Omega)}&=\big\| (\Delta)^{-1}\big[\nabla \cdot \big\{\nabla \cdot \g+\h+\boldsymbol{\vartheta}\mathds{1}_{\omega}-(\wi{\u} \cdot \nabla)\v-(\v \cdot \nabla)\wi{\u}\\&\quad -\beta|\wi{\u}|^{2}\v-2 \beta (\v \cdot \wi{\u}) \wi{\u}\big\} \big] \big\|_{\mathrm{L}^2(\Omega)} \\&
			\leq C\big\| \big[\nabla \cdot \big\{\nabla \cdot \g+\h+\boldsymbol{\vartheta}\mathds{1}_{\omega}-\nabla \cdot (\wi{\u} \otimes{\v})- \nabla \cdot (\v \otimes\wi{\u}) \\&\quad-\beta|\wi{\u}|^{2}\v-2 \beta (\v \cdot \wi{\u}) \wi{\u}\big\} \big] \big\|_{\mathbb{H}^{-2}(\Omega)} \\&
			\leq  C \big(\| \g \|_{\mathrm{L}^2(\Omega)^{d^2}}+\|\h  \|_{\mathrm{H}^{-1}(\Omega)^{d}}+\| \boldsymbol{\vartheta}\mathds{1}_{\omega} \|_{\mathrm{H}^{-1}(\Omega)^{d}}+\| \wi{\u}\otimes \v \|_{\H}+\| |\wi{\u}|^2 |\v| \|_{\V'} \big)\\&
			\leq C \big\{ \|\g \|_{\mathrm{L}^2(\Omega)^{d^2}}+\|\h \|_{\mathrm{H}^{-1}(\Omega)^{d}}+\|\boldsymbol{\vartheta} \|_{\mathrm{L}^2(\omega)^{d}}+\|\wi{\u}\|_{\mathrm{L}^{4}(\Omega)^{d}} \|\v\|_{\mathrm{L}^{4}(\Omega)^{d}} +\||\wi{\u}|^{2}|\v | \|_{\mathrm{L}^{\frac{4}{3}}(\Omega)^{d}}\big\}
			\\&
			\leq C \big\{ \|\g \|_{\mathrm{L}^2(\Omega)^{d^2}}+\|\h \|_{\mathrm{H}^{-1}(\Omega)^{d}}+\|\boldsymbol{\vartheta} \|_{\mathrm{L}^2(\omega)^{d}}+\|\wi{\u}\|_{\mathrm{L}^{4}(\Omega)^{d}} \|\v\|_{\V}+\||\wi{\u}||\v |\|_{\H}\|\wi{\u}\|_{\mathrm{L}^{4}(\Omega)^d} \big\}.
		\end{align*}
		Taking the square  on the both sides in the above equation and then integrating the resulting estimate from $0$  to $T$, we have 
		\begin{align}\label{a623}
			\int_{0}^{T} \|q(t)\|_{\mathrm{L}^2(\Omega)}^2 \d t &\leq C \bigg\{\int_{0}^{T} \|\g(t) \|_{\mathrm{L}^2(\Omega)^{d^2}}^2 \d t+  \int_{0}^{T}\|\h (t) \|_{\mathrm{H}^{-1}(\Omega)^{d}}^2 \d t+  \int_{0}^{T} \|\boldsymbol{\vartheta} (t)\|_{\mathrm{L}^2(\omega)^{d}}^2 \d t\no\\&\quad+  \sup_{t\in[0,T]} \| \wi{\u}(t)\|_{\mathrm{L}^4(\Omega)^d}^2\bigg(\int_{0}^{T} \|\v(t)\|_{\V}^2 \d t+  \int_{0}^{T} \||\wi{\u}(t)||\v (t)|\|_{\H}^2 \d t\bigg)\bigg\}.
		\end{align}
		Substituting  the estimate \eqref{a62} into \eqref{a623} leads to
		\begin{align}\label{a624}
			\|q\|_{\mathrm{L}^2(Q)}^2 \leq C\bigg(\|\v_0\|_{\H}^2 + \|\g\|_{\mathrm{L}^2(Q)^{d^2}}^2+ \|\h\|_{\mathrm{L}^2(0,T;\mathrm{H}^{-1}(\Omega)^{d})}^2 +
			\|\boldsymbol{\vartheta}\|_{\mathrm{L}^2(Q_{\omega})^{d}}^2  \bigg).
		\end{align}
		Hence, from \eqref{a62}, \eqref{a621} and \eqref{a624}, we derive the existence of a solution
		\begin{align*}
			\v \in \C([0,T];\H) \cap \mathrm{L}^2(0,T;\V), \ \ \text{and} \ \ q \in \mathrm{L}_{0}^2(Q),
		\end{align*}
		and \eqref{a7} follows from \eqref{a62}, \eqref{a621} and \eqref{a624}.
		\vskip 0.2cm
		\noindent
		\textbf{Uniqueness and stability:} 
		Since system \eqref{a5} is linear, we can see from an estimate similar to \eqref{a7} (see \eqref{a71}) that the solution  $(\v,q)$ depends continuously on the initial data, which leads to the Lipschitz stability of the solution $(\v,q)$  in \eqref{a71}, and the uniqueness is immediate. 
	\end{proof}
	We remark that the significance of this result will be highlighted in sections \ref{sec4} and \ref{sec5}, especially for proving the null controllability of the linearized system \eqref{a5}. 
	\section{Carleman Estimates and observability inequality}\label{sec3}\setcounter{equation}{0}
	The exact controllability problem for a linear evolution system can often be proven by examining the observability inequality of the associated adjoint system.
	
	In this section, we derive a Carleman estimate and an observability inequality for the adjoint system \eqref{a12} corresponding to the linearized system \eqref{a5}, which is necessary to study the exact controllability of the trajectories of the CBF equations. The only way we know to obtain such an estimate and inequality is to use a global Carleman estimate for the Stokes system.  We begin this section by addressing the weight functions that will be used throughout the article. Other weight classes will be presented in the sequel as needed.

	\subsection{Weight functions}
	We proceed to construct appropriate weight functions. Let us recall that  $\omega \subset \Omega$ is a non-empty open set.
	\begin{lemma}\label{lem1}
		Let $\omega$ be an arbitrary fixed subdomain of $\Omega$. Then, there exists a function $\eta \in \C^2(\overline{\Omega})$ such that
		\begin{align*}
			\eta(x) > 0, \ \ \forall \ x \in \Omega, \ \ \ \ \eta|_{\partial \Omega}=0, \ \ \ \ |\nabla \eta (x)|>0, \ \ \forall \ x \in \overline{ \Omega \backslash{\omega}}.
		\end{align*}
	\end{lemma}
	
	A proof of this lemma can be found in  \cite[Lemma 1.1]{FIm}. Let us define $\ell \in \C^\infty([0,T])$ be such that
	\begin{align*}
		\ell(t)=t \ \ \text{on} \ \bigg[0,\frac{T}{4}\bigg], \ \ \ell(t)=T-t \ \ \text{on} \ \bigg[\frac{3T}{4},T\bigg], \ \ \ell(t) \geq \frac{T}{4} \ \ \text{on} \ \bigg[\frac{T}{4}, \frac{3T}{4}\bigg].
	\end{align*}
	Using the function $\eta$ constructed in Lemma \ref{lem1}, we introduce the weight functions as
	\begin{align}
		\psi(x,t)&= \xi(x,t)-\frac{e^{\lambda (\|\eta\|_{\mathrm{L}^\infty}+m_{2})}}{\ell^{4}(t)}, \ \ \ \  \ \ \ \ \xi(x,t)=\frac{e^{\lambda (\eta(x)+m_1)}}{\ell^{4}(t)},\label{o5}\\
		\check{\psi}(t)&=\min_{x \in \overline{\Omega}}\psi(x,t)=\psi(t)|_{\partial \Omega}=\frac{e^{\lambda m_1}-e^{\lambda (\|\eta\|_{\mathrm{L}^\infty}+m_{2})}}{\ell^{4}(t)}, \label{10}\\
		\check{\xi}(t)&=\min_{x \in \overline{\Omega}}\xi(x,t)=\xi(t)|_{\partial \Omega}=\frac{e^{\lambda m_1}}{\ell^{4}(t)},\label{11}
	\end{align}
	where $\lambda\geq 1$ and the constants $m_1$ and $m_2$ are chosen such that $m_1 \leq m_2$ (see  \cite[Lemma 4.1]{Puel}), and there exists a constant $C>0$ such that 
	\begin{align*}
		\bigg|\frac{\partial \psi}{\partial t}\bigg| \leq C \xi^\frac{5}{4}, \ \ \ \text{and} \ \ \ \bigg|\frac{\partial^{2} \psi}{\partial t^2}\bigg| \leq C \xi^\frac{3}{2}.
	\end{align*} 
	We notice that 
	\begin{equation}\label{12}
		\left\{
		\begin{aligned}
			&\frac{1} {\check{\xi}(t)^{\frac{1}{4}}} \frac{\partial \check{\psi}(t)}{\partial t}\leq C \xi(t), \ \ \  \frac{1} {\check{\xi}(t)^{\frac{5}{4}}} \frac{\partial \check{\xi}(t)}{\partial t} \leq C, \ \ \text{for all} \ t \in [0,T], \\
			& \check{\psi}(t)\leq \psi(x,t), \ \ \ \check{\xi}(t) \leq \xi(x,t), \ \ \ \check{\xi}(t) \geq C_0>0, \  \text{for all} \ (x,t) \in \Omega \times (0,T).
		\end{aligned}
		\right.
	\end{equation} 
	We define the space $\mathrm{H}^{\frac{1}{4},\frac{1}{2}}(\Sigma)$ using a classical notation as
	\begin{align*} \mathrm{H}^{\frac{1}{4},\frac{1}{2}}(\Sigma)=\mathrm{H}^{\frac{1}{4}}(0,T; \mathrm{L}^2(\partial \Omega)) \cap \mathrm{L}^2(0,T; \mathrm{H}^{\frac{1}{2}}(\partial \Omega)).
	\end{align*}
	We remark that Fursikov and Imanuvilov  first introduced functions of this type in \cite{FIm} to obtain Carleman inequalities for the heat equations. We also state that these  weight functions were already defined by J.-P. Puel in the work \cite{Puel} in order to obtain a Carleman estimate for the Stokes system.
	\subsection{Carleman estimates} In the following subsection, we deduce a Carleman estimate for the adjoint system \eqref{a12} associated with the linearized system \eqref{a5}. This will provide a null controllability result for the linearized CBF equations \eqref{a5}. We closely follow similar ideas of Carleman estimates for the Stokes system with Dirichlet conditions, which were proved by J.-P. Puel in \cite{Puel}. To this end, let us consider the following Stokes system:
	\begin{equation}\label{o1}
		\left\{
		\begin{aligned}
			\frac{\partial\y}{\partial t}- \Delta\y+\nabla q&=\k, \ \text{ in }  Q \\ \nabla\cdot\y&=0,  \ \text{ in }  Q,\\
			\y&=\mathbf{0},  \ \text{ on }  \Sigma,  \\
			\y(0)&=\y_0,  \text{ in }  \Omega.
		\end{aligned}
		\right.
	\end{equation}
	It is well known (see \cite{RT}) that if $\y_0 \in \H$ and $ \k \in \mathrm{L}^2\big(0,T;\mathrm{H}^{-1}(\Omega)^d\big)$, where $\mathrm{H}^{-1}(\Omega)$ is the dual space of $\mathrm{H}_0^{1}(\Omega)$, there exist a unique solution $(\y,q)$ of the system \eqref{o1} such that 
	\begin{align*}
		\y \in \C([0,T];\H) \cap \mathrm{L}^2(0,T;\V), \ \ \ \text{and} \ \ \ q \in \mathrm{L}_0^2(Q).
	\end{align*}
	Moreover, from the classical regularity result, if $\y_0 \in \V$ and $ \k \in \mathrm{L}^2(Q)^d$, we have
	\begin{align*}
		\y \in \C([0,T];\V) \cap \mathrm{L}^2(0,T;\H^2(\Omega)\cap \V), \ \ 	\frac{\partial\y}{\partial t} \in \mathrm{L}^2(0,T;\H),  \ \ \ \text{and} \ \ \ \nabla q \in \mathrm{L}^2(Q)^d,
	\end{align*}
	and in particular
	\begin{align}\label{o2}
		\bigg	\|\frac{\partial\y}{\partial t}\bigg\|_{\mathrm{L}^2(0,T;\H)}+\|\y\|_{\mathrm{L}^{\infty}(0,T;\V)}+\|\y\|_{\mathrm{L}^2(0,T;\H^2(\Omega))}+\|\nabla q\|_{\mathrm{L}^2(Q)^d} \leq C\big(\|\y_0\|_{\V}+\|\k\|_{\mathrm{L}^2(Q)^d} \big).
	\end{align}
	Let us define $\z:=\nabla \times \y=\text{curl} \ \y$. Taking the curl of the equation \eqref{o1}, we get
	\begin{equation}\label{o3}
		\left\{
		\begin{aligned}
			\frac{\partial\z}{\partial t}- \Delta\z&=\nabla \times \k , \ \text{ in }  Q, \\
			\z(x,0)&=\z_0(x)=\nabla \times \y_0(x), \ \ \text{ in }  \Omega.
		\end{aligned}
		\right.
	\end{equation}
	As $\nabla \cdot \y=0$, we also have for a.e. $t \in (0,T)$
	\begin{equation}\label{o4}
		\left\{
		\begin{aligned}
			-\Delta\y(t)&=\nabla \times \z(t) , \ \text{ in } \Omega, \\
			\y(t)&=\mathbf{0}, \ \text{ on }  \partial\Omega.
		\end{aligned}
		\right.
	\end{equation}
	We notice that $ \y=\mathbf{0} \ \text{ on } \partial\Omega$ implies 
	$${\nabla \y=(\nabla \y \cdot \boldsymbol{n})\boldsymbol{n},  \ \text{ on } \ \partial\Omega.}$$
	Thus, $\z_{|_ {\partial\Omega}}$ can be expressed in terms of $\nabla \y \cdot \boldsymbol{n}$.
	
	The vector function $\z$ verifies a system of non-homogeneous heat equations \eqref{o3}. We can use the global Carleman estimate from \cite{ImpY} or an improvement established in \cite{ImpY1}.	For non-homogeneous parabolic equations, the following result of the global Carleman estimate has been established in \cite{ImpY1}.
	\begin{theorem}[Theorem 2.2, \cite{ImpY1}]
		Let $\omega$ be a non-empty open subset of $\Omega$ and $T>0$. Let us consider the weights $\eta,\psi$ and $ \xi$ defined in Lemma \ref{lem1} and \eqref{o5}. Then, there exist $s_0 \geq 1, \ \lambda_0 \geq 1$ and $C>0$ such that for  every $s \geq s_0$ and $\lambda \geq \lambda_0$, we have
		\begin{align}\label{o6}
			\int_{Q}e^{2s\psi}\bigg(\frac{|\nabla \z|^2}{s \xi }+s \lambda^2 \xi |\z|^2 \bigg)\d x \d t  
			& \leq C \bigg(s^{-\frac{1}{2}}\|\xi^{-\frac{1}{4}}e^{s\psi}\z \|_{\mathrm{H}^{\frac{1}{4},\frac{1}{2}}(\Sigma)}^2+\int_{Q}e^{2s\psi}|\k|^2 \d x \d t\bigg)\no \\&\quad+Cs \lambda^2 \int_{Q_{\omega}}e^{2s\psi} \xi |\z|^2 \d x \d t.
		\end{align}
	\end{theorem}
	The vector function $\y$ is a solution of the elliptic equation \eqref{o4} with right hand side $\nabla \times \z$. We can apply a result of \cite{ImP} which gives a Carleman estimate for the elliptic equation \eqref{o4} with the weight function  for fixed $ t \in (0,T)$,
	\begin{align}\label{o61}
		\gamma (x)=e^{\lambda(\eta(x)+m_{1})}.
	\end{align}
	\begin{theorem}[Theorem A.1, \cite{ImP}]
		Let $\omega$ be a non-empty open subset of $\Omega$ and  $\gamma$ be the weight function defined in \eqref{o61}. Then,
		there exists $\tau_0 \geq 1, \ \lambda_0 \geq 1$ and $C>0$ such that for  every $\tau \geq \tau_0$ and $\lambda \geq \lambda_0$, the function $\y(\cdot)$ satisfies 
		\begin{align}\label{o7}
			&\int_{\Omega}e^{2\tau\gamma }\big(|\nabla \y(t)|^2+\tau^2 \lambda^2 \gamma^2 |\y(t)|^2 \big)\d x  \no \\
			& \leq C \bigg(\tau \int_{\Omega}e^{2 \tau \gamma}\gamma|\z (t)|^2 \d x+\tau^2 \lambda^2\int_{\omega}e^{2 \tau \gamma }\gamma^2|\y(t)|^2 \d x\bigg),
		\end{align}
		for a.e. $t \in (0,T)$.
	\end{theorem}
	Let us take $\tau=\frac{s}{\ell^{4}(t)}$ and multiply the expression \eqref{o7} by $\lambda^2 \exp\bigg(-2s\frac{ e^{ \lambda (\|\eta\|_{\mathrm{L}^\infty}+m_{2})}}{\ell^{4}(t)}\bigg)$, and then integrate the resulting estimate from $0$ to $ T$, we obtain 
	\begin{align}\label{o8}
		&\int_{Q}e^{2s\psi}\big(\lambda^2|\nabla \y|^2+s^2 \lambda^4 \xi^2|\y|^2\big) \d x\d t\no \\
		& \leq C \bigg(s \lambda^2\int_{Q}e^{2s\psi}\xi|\z|^2 \d x \d t+s^2 \lambda^4 \int_{Q_{\omega}}e^{2s\psi} \xi^2 |\y|^2 \d x \d t\bigg).
	\end{align}
	Combining the estimates \eqref{o6} and \eqref{o8}, we arrive at
	\begin{align}\label{o9}
		&\int_{Q}e^{2s\psi}\bigg(\frac{|\nabla \z|^2}{s \xi }+s \lambda^2 \xi |\z|^2+\lambda^2|\nabla \y|^2+s^2 \lambda^4 \xi^2|\y|^2\bigg) \d x \d t \no \\
		& \leq C \bigg(s^{-\frac{1}{2}}\bigg\|\xi^{-\frac{1}{4}}e^{s\psi}\frac{\partial \y}{\partial \boldsymbol{n}} \bigg\|_{\mathrm{H}^{\frac{1}{4},\frac{1}{2}}(\Sigma)}^2+\int_{Q}e^{2s\psi}|\k|^2 \d x \d t\bigg)\no\\&\quad+C \int_{Q_{\omega}}e^{2s\psi}\big(s \lambda^2 \xi |\z|^2+ s^2 \lambda^4 \xi^2|\y|^2\big) \d x \d t.
	\end{align}
	We will now use the classical regularity estimates for the solutions of the Stokes system to estimate the boundary term $\left\|\xi^{-\frac{1}{4}}e^{s\psi}\frac{\partial \y}{\partial \boldsymbol{n}} \right\|_{\mathrm{H}^{\frac{1}{4},\frac{1}{2}}(\Sigma)}^2$. To do this, let us introduce the following functions:
	\begin{align*}
		\y^*=\check{\xi}(t)^{-\frac{1}{4}}e^{s\check{\psi}(t)}\y,  \ \ \ \text{and} \ \ \   q^*=\check{\xi}(t)^{-\frac{1}{4}}e^{s\check{\psi}(t)} q,
	\end{align*}
	where $\check{\psi}$ and $\check{\xi}$ are defined in \eqref{10} and \eqref{11}, respectively.
	Then, they satisfy
	\begin{equation}\label{o12}
		\left\{
		\begin{aligned}
			\frac{\partial\y^*}{\partial t}- \Delta\y^*+\nabla q^*&=\frac{e^{s\check{\psi}}}{\check{\xi}^{\frac{1}{4}}}\k+\frac{s} {\check{\xi}^{\frac{1}{4}}} \frac{\partial \check{\psi}}{\partial t}e^{s\check{\psi}}\y-\frac{1} {4\check{\xi}^{\frac{5}{4}}} \frac{\partial \check{\xi}}{\partial t}e^{s\check{\psi}}\y, \  \text{ in }  Q, \\ \nabla\cdot\y^*&=0, \  \text{ in }  Q, \\ 
			\y^*&=\mathbf{0}, \ \text{ on }  \Sigma, \\  \y^*(0)&=\mathbf{0},  \ \text{ in }  \Omega.
		\end{aligned}
		\right.
	\end{equation}
	Using the estimate \eqref{12} and	regularity result for the above system (cf. \cite{RT}) yield
	\begin{align}\label{o121}
		\bigg\|\frac{\partial\y^*}{\partial t} \bigg\|_{\mathrm{H}^1(0,T;\H)}+\|\y^*\|_{\mathrm{L}^2(0,T;\H^2(\Omega))}^2 \leq C\big(\|e^{s \psi} \k\|_{\mathrm{L}^2(Q)^d}^2+ s^{2}\|\xi e^{s \psi}\y\|_{\mathrm{L}^2(0,T;\H)}^2\big).
	\end{align}
	On the other hand, using \cite[Proposition 3.2]{ImpY}, we have
	\begin{align}\label{o122}
		\bigg	\|\frac{\partial\y^*}{\partial \boldsymbol{n}}\bigg\|_{\mathrm{H}^{\frac{1}{4},\frac{1}{2}}(\Sigma)}^2=\bigg\|\frac{e^{s\check{\psi}}}{\check{\xi}^{\frac{1}{4}}}\frac{\partial\y}{\partial \boldsymbol{n}}\bigg\|_{\mathrm{H}^{\frac{1}{4},\frac{1}{2}}(\Sigma)}^2 \leq C\|\y^*\|_{\mathrm{H}^{1,2}(Q)}^2,
	\end{align}
	where $$\mathrm{H}^{1,2}(Q)=\mathrm{H}^1(0,T;\H) \cap \mathrm{L}^2(0,T;\H^2(\Omega)).$$ 
	From \eqref{o121} and \eqref{o122}, we deduce 
	\begin{align}\label{o13}
		s^{-\frac{1}{2}}	\bigg\|\frac{e^{s\check{\psi}}}{\check{\xi}^{\frac{1}{4}}}\frac{\partial\y}{\partial \boldsymbol{n}}\bigg\|_{\mathrm{H}^{\frac{1}{4},\frac{1}{2}}(\Sigma)}^2 \leq C\left(s^{-\frac{1}{2}}\|e^{s \psi} \k\|_{\mathrm{L}^2(Q)^d}^2+ s^{\frac{3}{2}}\|\xi e^{s \psi}\y\|_{\mathrm{L}^2(0,T;\H)}^2\right).
	\end{align}
	Substituting the estimate \eqref{o13} into \eqref{o9} and then taking $s_0$ large enough such that for $s \geq s_0$, we can absorb the term $s^{\frac{3}{2}}\|\xi e^{s \psi}\y\|_{\mathrm{L}^2(0,T;\H)}^2 $ from the left hand-side, and we obtain the first Carleman estimate for the Stokes system ($\z= \nabla \times \y$)
	\begin{align}\label{o14}
		&\int_{Q}e^{2s\psi}\bigg(\frac{|\nabla \z|^2}{s \xi }+s \lambda^2 \xi |\z|^2+\lambda^2|\nabla \y|^2+s^2 \lambda^4 \xi^2 |\y|^2\bigg) \d x \d t \no \\
		& \leq C \bigg(\int_{Q}e^{2s\psi}|\k|^2 \d x \d t+ \int_{Q_{\omega}}e^{2s\psi}\big(s \lambda^2 \xi |\z|^2+s^2\lambda^4 \xi^2 |\y|^2\big) \d x \d t\bigg).
	\end{align}
	Now,  we want to get ride of the local term $Cs \lambda^2\int_{Q_{\omega}}e^{2s\psi} \xi |\z|^2 \d x\d t$ from the right hand side in \eqref{o14}. First of all, we can find a non-empty open subset $\omega_0$ of $\Omega$ such that $\overline{\omega_0} \subset \omega$ and $|\nabla\eta(x)| >0$ in $ \overline{\Omega \backslash\omega_0}$. Thus, we can replace $\omega$ by $\omega_0$ in the second term of the right-hand side of \eqref{o14} to deduce
	\begin{align}\label{o15}
		&\int_{Q}e^{2s\psi}\bigg(\frac{|\nabla \z|^2}{s \xi }+s \lambda^2 \xi |\z|^2+\lambda^2|\nabla \y|^2+s^2 \lambda^4 \xi^2 |\y|^2\bigg) \d x \d t \no \\
		& \leq C \bigg(\int_{Q}e^{2s\psi}|\k|^2 \d x \d t+ s \lambda^2 \int_{Q_{\omega_{0}}}e^{2s\psi} \xi |\z|^2\d x \d t+s^2\lambda^4  \int_{Q_{\omega}}e^{2s\psi}\xi^2 |\y|^2 \d x \d t\bigg).
	\end{align}
	Now, our aim is to eliminate the term $Cs \lambda^2\int_{Q_{\omega_{0}}}e^{2s\psi} \xi |\z|^2 \d x\d t$ in \eqref{o15}.
	Let us take a cut-off function $\theta \in \C_{0}^{\infty}(\omega)$ such that
	\begin{align*}
		\theta \equiv 1, \  \text{ in }  \omega_{0}, \ \text{ and } \  \ 0 \leq \theta \leq 1, \  \text{ in }  \omega.
	\end{align*}
	Then, we have
	\begin{align*}
		&C	s \lambda^2\int_{Q_{\omega_{0}}}e^{2s\psi} \xi |\z|^2 \d x\d t \leq 	Cs \lambda^2\int_{Q_{\omega}}e^{2s\psi}\theta \xi |\z|^2 \d x\d t
		\\&=C s \lambda^2\int_{Q_{\omega}} e^{2s\psi}\theta \xi \z  \cdot (\nabla \times\y)\d x\d t
		=C s \lambda^2\int_{Q_{\omega}}\nabla \times \big(e^{2s\psi}\theta \xi \z \big) \cdot\y\d x\d t \\&\leq C s \lambda^2\int_{Q_{\omega}}e^{2s\psi}\theta \xi |\nabla  \z|| \y| \d x\d t + C s \lambda^2\int_{Q_{\omega}}e^{2s\psi} \xi |\y| \big(2s \theta \lambda\xi |\z|+\xi |\z|+\lambda \xi\theta |\z|\big) \d x\d t
		\\& \leq
		C\bigg(\int_{Q_{\omega}}\frac{e^{2s\psi}}{s \xi}|\nabla \z|^2\d x \d t\bigg)^\frac{1}{2}\bigg(\int_{Q_{\omega}}s^3 \lambda^4 \xi^3e^{2s\psi}| \y|^2\d x \d t\bigg)^\frac{1}{2}\\&\quad+C\bigg(\int_{Q_{\omega}}s \lambda^2 \xi e^{2s\psi}| \z|^2\d x \d t\bigg)^\frac{1}{2}\bigg(\int_{Q_{\omega}}s^3 \lambda^4 \xi^3e^{2s\psi}| \y|^2\d x \d t\bigg)^\frac{1}{2}
		\\&\quad+C\bigg(\int_{Q_{\omega}}s \lambda^2 \xi e^{2s\psi}| \z|^2\d x \d t\bigg)^\frac{1}{2}\bigg(\int_{Q_{\omega}}s \lambda^2 \xi e^{2s\psi}| \y|^2\d x \d t\bigg)^\frac{1}{2}
		\\&\quad+C\bigg(\int_{Q_{\omega}}s \lambda^2 \xi e^{2s\psi}| \z|^2\d x \d t\bigg)^\frac{1}{2}\bigg(\int_{Q_{\omega}}s \lambda^4 \xi e^{2s\psi}| \y|^2 \d x \d t\bigg)^\frac{1}{2} \\& \leq
		\frac{1}{2}\int_{Q_{\omega}}\frac{e^{2s\psi}}{s \xi}|\nabla \z|^2\d x \d t+\frac{1}{2} \int_{Q_{\omega}}s \lambda^2 \xi e^{2s\psi}| \z|^2\d x \d t+C\int_{Q_{\omega}}s^3 \lambda^4 \xi^3e^{2s\psi}| \y|^2\d x \d t,
	\end{align*}
	where we have used the Cauchy-Schwarz and Young's inequalities.
	Substituting the above estimate in \eqref{o15}, we arrive at
	\begin{align}\label{o18}
		\no	&\int_{Q}e^{2s\psi}\bigg(\frac{|\nabla \z|^2}{s \xi }+s \lambda^2 \xi |\z|^2+\lambda^2|\nabla \y|^2+s^2 \lambda^4 \xi^2 |\y|^2\bigg) \d x \d t \no \\
		& \leq C \bigg(\int_{Q}e^{2s\psi}|\k|^2 \d x \d t+ \int_{Q_{\omega}}e^{2s\psi}s^2\lambda^4 \xi^2 |\y|^2 \d x \d t\bigg)+\frac{1}{2}\int_{Q_{\omega}}\frac{e^{2s\psi}}{s \xi}|\nabla \z|^2\d x \d t\no\\&\quad+\frac{1}{2} \int_{Q_{\omega}}s \lambda^2 \xi e^{2s\psi}| \z|^2\d x \d t+C\int_{Q_{\omega}}s^3 \lambda^4 \xi^3e^{2s\psi}| \y|^2\d x \d t.
	\end{align} 
	From \eqref{o18}, we obtain the Carleman estimate for the Stokes system as follows:
	\begin{theorem}
		Let $\omega$ be a non-empty open subset of $\Omega$ and $T>0$. Let us consider the weight functions $\eta, \psi$ and $\xi$ defined in Lemma \ref{lem1} and \eqref{o5}. Then there exists $s_0 \geq 1, \ \lambda_0 \geq 1$ and $C>0$ such that for $s \geq s_0$ and $\lambda \geq \lambda_0$ and every solution $\y$ of the Stokes system \eqref{o1}, we have
		\begin{align}\label{o19}
			&\int_{Q}e^{2s\psi}\bigg(\frac{|\nabla \z|^2}{s \xi }+s \lambda^2 \xi |\z|^2+\lambda^2|\nabla \y|^2+s^2 \lambda^4 \xi^2 |\y|^2\bigg) \d x \d t \no \\
			& \leq C \bigg(\int_{Q}e^{2s\psi}|\k|^2 \d x \d t+s^3 \lambda^4 \int_{Q_{\omega}}e^{2s\psi} \xi^3 |\y|^2 \d x \d t\bigg).
		\end{align}
	\end{theorem}
	Note  that the adjoint system \eqref{a12} can be viewed as the previous	Stokes system \eqref{o1} with $t$ and $\k$ replaced by $T-t$ and $$
	\k^*=(D \vr) \wi{\u}-\alpha\vr-\beta|\wi{\u}|^{2}\vr-2 \beta (\vr \cdot \wi{\u}) \wi{\u},$$
	respectively, with the hypothesis that $\wi{\u} \in \mathrm{L}^\infty(Q)^d$. Thus, it is immediate that
	\begin{align*}
		\|e^{s\psi}\k^*\|_{\mathrm{L}^2(Q)^d}^2 & \leq 2\|\wi{\u}\|_{\mathrm{L}^\infty(Q)^d}^2\|e^{s\psi} \nabla \vr\|_{\mathrm{L}^2(0,T;\H)}^2 +4\alpha^2\|e^{s\psi} \vr\|_{\mathrm{L}^2(0,T;\H)}^2 \\&\quad+36\beta^2 \|\wi{\u}\|_{\mathrm{L}^\infty(Q)^d}^4\|e^{s\psi}  \vr\|_{\mathrm{L}^2(0,T;\H)}^2\nonumber\\&\leq C\|e^{s\psi} \nabla \vr\|_{\mathrm{L}^2(0,T;\H)}^2 +C\|\lambda\xi e^{s\psi}  \vr\|_{\mathrm{L}^2(0,T;\H)}^2.
	\end{align*}
	Now, applying the Carleman estimate \eqref{o19}  and by choosing $\lambda_0$ (where $\lambda \geq \lambda_0$) large enough, we can absorb the terms in the left-hand side and we have the following Carleman estimate for the adjoint state $\vr$ (see \eqref{a12}):
	\begin{align}\label{o20}
		\no	&\int_{Q}e^{2s\psi}\bigg(\frac{|\nabla(\nabla \times \vr)|^2}{s \xi }+s \lambda^2 \xi |\nabla \times \vr|^2+\lambda^2|\nabla \vr|^2+s^2 \lambda^4 \xi^2 |\vr|^2\bigg) \d x \d t \no \\
		& \leq C s^3 \lambda^4\int_{Q_{\omega}} \xi^3e^{2s\psi}| \vr|^2\d x \d t.
	\end{align}
	
	From now on, let us fix $s \geq s_0$ and $\lambda \geq \lambda_0$ such that \eqref{o20} holds. We consider the following weight functions in order to deduce a Carleman estimate that does not vanish at $t=0$:
	\begin{equation}\label{21}
		\left\{
		\begin{aligned}
			\wi{\psi}(t)=\psi(t), \ \ \text{if} \ t \in [T/2,T], \ \ \ \wi{\psi}(t)=\psi(T/2), \ \ \text{if} \ t \in [0,T/2], \\
			\wi{\xi}(t)=\xi(t), \ \ \text{if} \ t \in [T/2,T], \ \ \ \ \ \wi{\xi}(t)=\xi(T/2), \ \ \text{if} \ t \in [0,T/2].
		\end{aligned}
		\right.
	\end{equation}
	We note that the new weight functions $\wi{\psi}$ and $\wi{\xi}$ are no longer degenerate in the neighborhood of $t=0$.  Moreover, by the construction of weight functions $\psi(x,t)$ and $\xi(x,t)$ defined in \eqref{o5}, one can see that $\wi{\psi} \leq 0$ for $t \in [0,T]$. Replacing $\psi$ and $\xi$ by $\wi{\psi}$ and $\wi{\xi}$, respectively, and using the standard energy estimates for the backward Stokes system, one can obtain the same Carleman estimate:
	\begin{align}\label{o23}
		\no	&\int_{Q}e^{2s\wi{\psi}}\bigg(\frac{|\nabla(\nabla \times \vr)|^2}{s \wi{\xi} }+s \lambda^2 \wi{\xi} |\nabla \times \vr|^2+\lambda^2|\nabla \vr|^2+s^2 \lambda^4 \wi{\xi}^2 |\vr|^2\bigg) \d x \d t \no \\
		& \leq C s^3 \lambda^4\int_{Q_{\omega}} \wi{\xi}^3e^{2s\wi{\psi}}| \vr|^2\d x \d t.
	\end{align}

	\subsection{Observability inequality}\label{subseco}
	In this subsection, we derive an observability inequality for the adjoint system \eqref{a12} associated with the linearized system \eqref{a5} around a trajectory $(\wi{\u},\wi{p})$ of \eqref{a2}, which leads to  the null controllability of the system \eqref{a5} at any time $T>0$. First, we present a remark which is helpful in deriving the observability inequality.
	\begin{remark}
		Let $\vr(T) \in \H$ and let $(\vr,\pi)$ be the corresponding solution of the adjoint system \eqref{a12}. We know that $\vr \in \C([0,T];\H) \cap \mathrm{L}^2(0,T;\V)$. Let $\chi \in \C^{1}([0,T])$ be such that $\chi(T)=0$. Then $(\wi{\vr},\wi{\pi}):=(\chi \vr,\chi \pi)$ solves the system:
		\begin{equation}\label{o2311}
			\left\{
			\begin{aligned}
				\mathcal{L}^{*} \wi{\vr}+\nabla \wi{\pi}&=-\chi'\vr,  \ \text{ in }  Q,\\   \nabla\cdot \wi{\vr}&=0, \ \text{ in } Q, 
				\\ \wi{\vr}&=\mathbf{0},  \ \text{ on }  \Sigma, \\
				\wi{\vr}(T)&=\mathbf{0},  \ \text{ in }  \Omega,
			\end{aligned}
			\right.
		\end{equation}
		where $\mathcal{L}^{*}$ is the adjoint operator of $\mathcal{L}$ defined in \eqref{a121}. Therefore, $(\wi{\vr},\wi{\pi})$ is a strong solution of \eqref{o2311} and we have
		\begin{align*}
			\vr(t) \in \H^2(\Omega)\cap \V,  \ \ \vr_{t}(t) \in \H,  \ \ \text{and} \ \	\pi(t) \in \mathrm{H}^1(\Omega)\cap\mathrm{L}^2_0(\Omega),
		\end{align*}
		for a.e. $t \in [0,T]$.
	\end{remark}
	\begin{theorem}\label{thmo}
		Let $\wi{\psi}$ and $\wi{\xi}$ be the weight functions defined in \eqref{21} and  $(\wi{\u},\wi{p})$ be the solution of \eqref{a2} satisfying \eqref{a2a}. Then, there exists a positive constant $C$ depending on $s, \lambda, \wi{\u}$ and $T$, such that every solution to the adjoint system \eqref{a12} satisfies:
		\begin{align}\label{o231}
			&\|\vr(0)\|_{\H}^2+\int_{Q} \wi{\xi}^2 e^{2s \wi{\psi}}| \vr|^2\d x \d t+\int_{Q}e^{2s \wi{\psi}}| \nabla \vr|^2\d x \d t
			\no\\&\leq C \int_{Q_{\omega}} \wi{\xi}^3e^{2s\wi{\psi}}| \vr|^2\d x \d t.
		\end{align}
	\end{theorem}
	\begin{proof}
		The proof of this lemma is classical and relies on the Carleman estimate \eqref{o23} and the dissipation properties of the solutions of the adjoint system  \eqref{a12}. This tells us that $\|\vr(t)\|_{\H}$ is an increasing function of $t$, so the presence of $\|\vr(0)\|_{\H}$ in the left-hand side of \eqref{o231} is not a surprise (see \cite{SG}). For completeness, we include the proof here (cf. \cite{CGIP,SG}).
		
		We start with an a-priori estimate for the adjoint system \eqref{a12}:
		\begin{align}\label{o232}
			\sup_{t\in[0,T/2]}\|\vr(t)\|_{\H}^2+\int_0^{T/2}\|\vr(t)\|_{\V}^2 \d t\leq C \int_{T/2}^{3T/4}\|\vr(t)\|_{\H}^2\d t,
		\end{align}
		where $C$ depends on $\Omega, T$ and $ \wi{\u}$. To prove this, let us take $\chi \in \C^{1}([0,T])$ such that
		\begin{align*}
			\chi=1 \ \ \text{in} \ [0,T/2], \ \ \ \ \chi\equiv 0  \ \text{ in } [3T/4,T], \ \ \ \text{and}  \ \ \ |\chi'|\leq C.
		\end{align*}
		Then, $(\chi \vr,\chi \pi)$  solves the system \eqref{o2311}, and we obtain the following energy estimate:
		\begin{align}\label{o233}
			\sup_{t\in[0,T]}\|\chi(t)\vr(t)\|_{\H}^2+	\int_0^T\|\chi(t)\vr(t)\|_{\V}^2\d t \leq C \int_0^T\|\chi'(t)\vr(t)\|_{\H}^2\d t, 
		\end{align}
		and \eqref{o232} follows from \eqref{o233}.  In $\Omega\times(0,T/2)$, \eqref{o232} implies
		\begin{align}\label{o234}
			&\|\vr(0)\|_{\H}^2+\int_{0}^{T/2}\int_{\Omega} \wi{\xi}^2 e^{2s \wi{\psi}}| \vr|^2\d x \d t+\int_{0}^{T/2}\int_{\Omega}e^{2s \wi{\psi}}| \nabla \vr|^2\d x \d t \no\\& \leq  C\left(\sup_{t\in[0,T/2]}\|\vr(t)\|_{\H}^2+	\int_0^{T/2}\|\vr(t)\|_{\V}^2\d t\right)\no\\&
			\leq C \int_{T/2}^{3T/4}\int_{\Omega} | \vr|^2\d x \d t
			\leq  \int_{Q}s^2\lambda^4\wi{\xi}^2e^{2s\wi{\psi}} | \vr|^2\d x \d t,
		\end{align}
	where we have used the fact $e^{2s \wi{\psi} }\leq e^{2s_{0} \wi{\psi}} \leq C$ for $\wi{\psi}\leq0$ and chosen  $0<C\leq s^2\lambda^4\wi{\xi}^2e^{2s\wi{\psi}}$ for $s \geq s_0 \geq1, \ \lambda \geq \lambda_0 \geq1$ in the last inequality. Applying the Carleman estimate \eqref{o23} into \eqref{o234}, we obtain a first estimate in $\Omega\times(0,T/2)$:
		\begin{align}\label{o235}
			&\|\vr(0)\|_{\H}^2+\int_{0}^{T/2}\int_{\Omega} \wi{\xi}^2 e^{2s \wi{\psi}}| \vr|^2\d x \d t+\int_{0}^{T/2}\int_{\Omega}e^{2s \wi{\psi}}| \nabla \vr|^2\d x \d t \no\\& \leq C s^3 \lambda^4 \int_{Q_{\omega}}\wi{\xi}^3e^{2s\wi{\psi}} | \vr|^2\d x \d t.
		\end{align}
		On the other hand, since $\psi=\wi{\psi}$ and $\xi=\wi{\xi}$ in $\Omega\times(T/2,T)$, we have
		\begin{align}\label{o236}
			&\int_{T/2}^{T}\int_{\Omega} \wi{\xi}^2 e^{2s \wi{\psi}}| \vr|^2\d x \d t+\int_{T/2}^{T}\int_{\Omega}e^{2s \wi{\psi}}| \nabla \vr|^2\d x \d t\no\\&=\int_{T/2}^{T}\int_{\Omega} {\xi}^2 e^{2s {\psi}}| \vr|^2\d x \d t+\int_{T/2}^{T}\int_{\Omega}e^{2s {\psi}}| \nabla \vr|^2\d x \d t \no\\&\leq \bigg(s^2 \lambda^4 \int_{Q} \xi^2 e^{2s {\psi}} |\vr|^2\d x \d t+\lambda^2\int_{Q}e^{2s\psi}|\nabla \vr|^2 \d x \d t\bigg) 
			\no\\&\leq C s^3 \lambda^4 \int_{Q_{\omega}}{\xi}^3e^{2s{\psi}} | \vr|^2\d x \d t
				\no\\& =C s^3 \lambda^4 \int_{0}^{T/2}\int_{\omega}{\xi}^3e^{2s{\psi}} | \vr|^2\d x \d t+C s^3 \lambda^4 \int_{T/2}^{T}\int_{\omega}{\xi}^3e^{2s{\psi}} | \vr|^2\d x \d t,
		\end{align}
		where we have used the Carleman estimate \eqref{o20}. By the definition of weight functions \eqref{o5} and \eqref{21}, from \eqref{o236}, we immediately get
		\begin{align}\label{o238}
			&\int_{T/2}^{T}\int_{\Omega} \wi{\xi}^2 e^{2s \wi{\psi}}| \vr|^2\d x \d t+\int_{T/2}^{T}\int_{\Omega}e^{2s \wi{\psi}}| \nabla \vr|^2\d x \d t
				\no\\&\leq C s^3 \lambda^4 \int_{0}^{T/2}\int_{\omega}| \vr|^2\d x \d t+C s^3 \lambda^4 \int_{T/2}^{T}\int_{\omega}\wi{\xi}^3e^{2s\wi{\psi}} | \vr|^2\d x \d t
				\no\\&\leq C s^3 \lambda^4 \int_{0}^{T/2}\int_{\omega} \wi{\xi}^3e^{2s\wi{\psi}} | \vr|^2\d x \d t+C s^3 \lambda^4 \int_{T/2}^{T}\int_{\omega}\wi{\xi}^3e^{2s\wi{\psi}} | \vr|^2\d x \d t
			\no\\&= C s^3 \lambda^4 \int_{Q_{\omega}}\wi{\xi}^3e^{2s\wi{\psi}} | \vr|^2\d x \d t,
		\end{align}
		which combined with \eqref{o235} leads to 
		\begin{align*}
			&\|\vr(0)\|_{\H}^2+\int_{Q} \wi{\xi}^2 e^{2s \wi{\psi}}| \vr|^2\d x \d t+\int_{Q}e^{2s \wi{\psi}}| \nabla \vr|^2\d x \d t
			\leq C s^3 \lambda^4 \int_{Q_{\omega}}\wi{\xi}^3e^{2s\wi{\psi}} | \vr|^2\d x \d t,
		\end{align*}
		and we get the desired observability inequality \eqref{o231}.
	\end{proof}

	\section{Null Controllability of the Linear System}\label{sec4}\setcounter{equation}{0}
	In this section, we prove the null controllability of the linearized CBF equations \eqref{a5} as a	consequence of the observability inequality (see Theorem \ref{thmo}). This result will be helpful to determine the local exact controllability of the system \eqref{a1} in the next section.
	
	Later discussion of the null controllability of the linearized CBF equations is motivated by the works \cite{CGIP,Puel}, in which the authors established the null controllability of the linearized Navier-Stokes system in both two and three dimensions.
	\subsection{Penalty method}\label{subsec4.1}
	With the help of observability inequality \eqref{o231}, one can prove the null controllability of the linearized system \eqref{a5}. We establish  the following null controllability result for the system \eqref{a5}. 
	\begin{theorem}\label{thmn1}
		If $\v_0 \in \H$ and the functions $\g$ and $\h$ satisfy
		\begin{align*}
			\int_{Q}e^{-2s\wi{\psi}}|\g|^2 \d x\d t <+\infty \ \ \  \text{and} \ \ \ \int_{0}^{T} \|e^{-s \wi{\psi}}\h\|_{\mathrm{H}^{-1}(\Omega)^d}^2  \d t <+\infty,
		\end{align*}
		respectively,	then there exists a control $\boldsymbol{\vartheta} \in \mathrm{L}^2(Q_{\omega})^d$ and a solution $\v$ of the system \eqref{a5} such that
		$$\v(\cdot,T)=\mathbf{0}.$$ 
	\end{theorem}
	\begin{proof}
		Let us consider, for each $\varepsilon >0$, the following optimal control problem:
		\begin{align}\label{a8}
			\min_{\boldsymbol{\vartheta} \in \mathrm{L}^2(Q_{\omega})^d} \mathcal{J}_{\varepsilon}(\boldsymbol{\vartheta}),
		\end{align}
		where
		\begin{align}\label{a9}
			\mathcal{J}_{\varepsilon}(\boldsymbol{\vartheta})=\frac{1}{2\varepsilon}\|\v(T)\|_{\H}^2+\frac{1}{2}\int_{Q_{\omega}}|\boldsymbol{\vartheta}|^2\d x\d t,
		\end{align}
		and $\v$ is the solution of  the system \eqref{a5} associated to $\boldsymbol{\vartheta}$.
		\vskip 0.2 cm

		\vskip 0.2 cm
		\noindent \textbf{Step 1.} \emph{Existence of an optimal control:} From the definition of the cost functional $\mathcal{J}_{\varepsilon}(\cdot)$ defined in \eqref{a9}, we can see that $\mathcal{J}_{\varepsilon}(\cdot) \geq 0$, there exists an infimum $I$ of $\mathcal{J}_{\varepsilon}(\cdot)$ over $\mathrm{L}^2(Q_{\omega})^d$, that is, 
		\begin{align*}
			0 \leq I:=	\inf_{\boldsymbol{\vartheta} \in \mathrm{L}^2(Q_{\omega})^d} \mathcal{J}_{\varepsilon}(\boldsymbol{\vartheta}) <+\infty.
		\end{align*}
		Since, $0 \leq I <+ \infty$, there exists a minimizing sequence $\{\boldsymbol{\vartheta}^{n}\} \in \mathrm{L}^2(Q_{\omega})^d$ such that
		\begin{align*}
			\lim_{n \rightarrow\infty} \mathcal{J}_{\varepsilon}(\boldsymbol{\vartheta}^{n})=\lim_{n \rightarrow\infty}\bigg(	\frac{1}{2\varepsilon}\|\v^n(T)\|_{\H}^2+\frac{1}{2}\int_{Q_{\omega}}|\boldsymbol{\vartheta}^n|^2\d x\d t\bigg) =I,
		\end{align*}
		where $\v^{n}(\cdot)$ is the unique solution of the linearized problem \eqref{a5} with the control  $\boldsymbol{\vartheta}^{n} \in \mathrm{L}^2(Q_{\omega})^d$ and initial data $\v^{n}(0) \in \H$.  Since,  $\mathbf{0} \in \mathrm{L}^2(Q_{\omega})^d$, without loss of generality, we may assume that $\mathcal{J}_{\varepsilon}(\boldsymbol{\vartheta}^{n}) \leq \mathcal{J}_{\varepsilon}(\mathbf{0})$, where $\v$ is the unique solution of \eqref{a5} corresponding to the control $\mathbf{0}$. Using  the definition of $\mathcal{J}_{\varepsilon}(\cdot)$, we deduce that 
		\begin{align*}
			\frac{1}{2\varepsilon}\|\v^n(T)\|_{\H}^2+\frac{1}{2}\int_{Q_{\omega}}|\boldsymbol{\vartheta}^n|^2\d x\d t \leq  \frac{1}{2\varepsilon}\|\v(T)\|_{\H}^2 < + \infty.
		\end{align*}
		From the above expression, it is clear that, there exists a constant $R>0$, large enough such that
		\begin{align}\label{a81}
			\frac{1}{2}\int_{Q_{\omega}}|\boldsymbol{\vartheta}^n|^2\d x\d t \leq R < + \infty.
		\end{align}
		Using the energy estimate obtained in Lemma \ref{thmwell}, we get
		\begin{align*}
			\|\v^n\|_{\mathrm{L}^\infty(0,T;\H)} + \|\v^n\|_{\mathrm{L}^2(0,T;\V)}  &\leq C\big(\|\v_0\|_{\H}+\|\g \|_{ \mathrm{L}^2(Q)^{d^2}}+\|\h \|_{ \mathrm{L}^2(0,T;\mathrm{H}^{-1}(\Omega)^d)}+\|\boldsymbol{\vartheta}^n\|_{\mathrm{L}^2(Q_{\omega})^d}\big)
			\no\\&\leq C\big(\|\v_0\|_{\H}+\|\g \|_{ \mathrm{L}^2(Q)^{d^2}}+\|\h \|_{ \mathrm{L}^2(0,T;\mathrm{H}^{-1}(\Omega)^d)}+R^{\frac{1}{2}}\big).
		\end{align*}
		We know that the dual of $\mathrm{L}^1(0,T,\H) $ is $\mathrm{L}^\infty(0,T,\H) $ and $\mathrm{L}^1(0,T,\H) $ is separable. Also the spaces $\mathrm{L}^2(0,T,\V) $ and $\mathrm{L}^2(Q_{\omega})^d$ are reflexive. By using the  Banach-Alaoglu theorem, we can extract subsequences still denoted by 	$\{\v^{n}\}$ and $\{\boldsymbol{\vartheta}^{n}\}$ such that
		\begin{equation*}
		\left\{
		\begin{aligned}
			& \v^{n} \xrightharpoonup{w^{\ast}} \v_{\varepsilon} \ \text{ in } \ \mathrm{L}^\infty(0,T,\H) ,\\
			&\v^{n} \xrightharpoonup{w} \v_{\varepsilon} \ \text{ in } \  \mathrm{L}^2(0,T;\V),\\
			& \boldsymbol{\vartheta}^{n} \xrightharpoonup{w} \boldsymbol{\vartheta}_{\varepsilon} \  \text{ in } \ \mathrm{L}^2(Q_{\omega})^d.
		\end{aligned}
		\right.
		\end{equation*}
		Moreover, $\v^n_t\xrightharpoonup{w}{\v_{\varepsilon}}_t\ \text{ in }\ \mathrm{L}^2(0,T;\V')$ (cf. \eqref{a621}). 
		Now, we show that  $(\v_{\varepsilon},\boldsymbol{\vartheta}_{\varepsilon})$  satisfies the system \eqref{a5}. Since, $\mathcal{L}$ is a linear operator defined in \eqref{a31}, from the above convergences, one can easily deduce 
		\begin{align*}
			& \mathcal{L}\v^{n} \xrightharpoonup{w} \mathcal{L}\v_{\varepsilon} \ \text{ in } \  \mathrm{L}^{2}(0,T;\V'). 
		\end{align*}
		On passing to limit $n \rightarrow \infty $ in the equation satisfied by $(\v^n,\boldsymbol{\vartheta}^{n})$,  we find that the limit $(\v_{\varepsilon},\boldsymbol{\vartheta}_{\varepsilon})$ satisfies \eqref{a5} in the weak sense.
		
		Finally, we show that $(\v_{\varepsilon},\boldsymbol{\vartheta}_{\varepsilon})$ is a minimizer, that is, $I =\mathcal{J}_{\varepsilon}(\boldsymbol{\vartheta}_{\varepsilon})$. Since the
		cost functional $\mathcal{J}_{\varepsilon}(\cdot)$ is continuous and convex on $\mathrm{L}^2(Q_{\omega})^d$ and $\mathrm{L}^\infty(0,T;\H) \cap \mathrm{L}^2(0,T;\V)$, it follows that $\mathcal{J}_{\varepsilon}(\cdot)$ is weakly lowersemicontinuous. That is, for a sequence $ \boldsymbol{\vartheta}^{n} \xrightharpoonup{w} \boldsymbol{\vartheta}_{\varepsilon} \  \text{ in } \ \mathrm{L}^2(Q_{\omega})^d,$
		we have
		\begin{align*}
			\mathcal{J}_{\varepsilon}(\boldsymbol{\vartheta}_{\varepsilon}) \leq \liminf_{n \rightarrow \infty} \mathcal{J}_{\varepsilon}(\boldsymbol{\vartheta}^n).
		\end{align*}
		Therefore, we obtain
		\begin{align*}
			I\leq	\mathcal{J}_{\varepsilon}(\boldsymbol{\vartheta}_{\varepsilon})  \leq \liminf_{n \rightarrow \infty} \mathcal{J}_{\varepsilon}(\boldsymbol{\vartheta}^n) = \lim_{n \rightarrow \infty} \mathcal{J}_{\varepsilon}(\boldsymbol{\vartheta}^n)=I,
		\end{align*}
		and hence $(\v_{\varepsilon},\boldsymbol{\vartheta}_{\varepsilon})$ is a minimizer of the problem \eqref{a8}.
		
		\vskip 0.2 cm
		\noindent \textbf{Step 2.} \emph{Uniqueness of the optimal control:}
		The optimal control problem \eqref{a8} has a solution $(\v_\varepsilon,\boldsymbol{\vartheta}_{\varepsilon})$ for any $\varepsilon >0$, and by the convexity of the cost functional and since the system \eqref{a5} is linear, this solution is even unique. 
		\vskip 0.2 cm
		\noindent \textbf{Step 3.} \emph{First order necessary conditions of optimality:} The necessary condition of minimum yields
		\begin{align}\label{a10}
			\bigg(\frac{1}{\varepsilon}\v_{\varepsilon}(T),\v_{\w}(T)\bigg)+\int_{Q_{\omega}}\boldsymbol{\vartheta}_{\varepsilon} \cdot\w \d x\d t=0,  \ \ \text{for all} \  \w \in \mathrm{L}^2(Q_{\omega})^d, 
		\end{align}
		where $\v_{\w}$  is, together with certain pressure $q_{\w}$, a solution of 
		\begin{equation}\label{a11}
			\left\{
			\begin{aligned}
				\mathcal{L} \v_{\w} +\nabla q_{\w}&=\w \mathds{1}_{\omega},  \ \text{ in }  Q, \\ \nabla\cdot\v_{\w}&=0, \ \text{ in }  Q,\\
				\v_{\w}&=\mathbf{0}, \ \text{ on }  \Sigma, \\ \v_{\w}(0)&=\mathbf{0},  \ \text{ in } \Omega.
			\end{aligned}
			\right.
		\end{equation}
		Here, the (affine) map's derivative as $\boldsymbol{\vartheta} \rightarrow\v(\boldsymbol{\vartheta})$ at the point $\boldsymbol{\vartheta}_{\varepsilon}$  in the direction
		$\w$ is represented by $\v_{\w}$. Then, the duality properties between $\v$ (see \eqref{a5}) and $\vr$ (see \eqref{a12})  yield
		\begin{align*}
			\bigg(\frac{1}{\varepsilon}\v_{\varepsilon}(T),\v_{\w}(T)\bigg)=\int_{Q_{\omega}}\vr \cdot\w \d x \d t, \ \ \text{for all} \  \w \in \mathrm{L}^2(Q_{\omega})^d,
		\end{align*}
		which combined with \eqref{a10}, provides
		\begin{align*}
			\int_{Q_{\omega}}(\boldsymbol{\vartheta}_{\varepsilon}+\vr)\cdot\w \d x\d t=0,  \ \ \text{for all} \  \w \in \mathrm{L}^2(Q_{\omega})^d.
		\end{align*}
		Consequently, we can identify $\boldsymbol{\vartheta}_{\varepsilon}$:
		\begin{align}\label{a13}
			\boldsymbol{\vartheta}_{\varepsilon}+\vr=\mathbf{0}, \  \text{ in } \ Q_{\omega}.
		\end{align}
		\vskip 0.2 cm
		\noindent \textbf{Step 4.} \emph{Null controllability of the linearized system:} 
		We are now ready to pass the limit as $\varepsilon \rightarrow 0$. To this end, by multiplying the state equation for $\v_{\varepsilon}$ (see \eqref{a5}) by $\vr$, we deduce 
		\begin{align*}
			\big(\mathcal{L} \v_{\varepsilon}+\nabla q_{\varepsilon},\vr\big)=\big(\nabla \cdot \g+\boldsymbol{\vartheta}_{\varepsilon} \mathds{1}_{\omega},\vr\big)+\langle\h,\vr \rangle. 
		\end{align*}
		On simplification, it is immediate that 
		\begin{align*}
			\frac{\d}{\d t}(\v_{\varepsilon},\vr)+\big(	\mathcal{L}^*\vr,\v_{\varepsilon}\big)=\big(\nabla \cdot \g+\boldsymbol{\vartheta}_{\varepsilon} \mathds{1}_{\omega},\vr\big)+\langle\h,\vr \rangle.
		\end{align*}
		Integrating the above equation with respect to time from $0$ to $T$, we find 
		\begin{align*}
			\big(\v_{\varepsilon}(T),\vr(T)\big)-\big(\v_0,\vr(0)\big)=-\sum_{i,j=1}^d \int_{Q}\g_{ij}\frac{\partial \vr_{i}}{\partial x_{j}} \d x\d t+\int_{Q_{\omega}}\boldsymbol{\vartheta}_{\varepsilon}\cdot \vr \d x\d t+\int_{0}^T\langle\h,\vr \rangle \d t.
		\end{align*}
		The optimality condition \eqref{a13} and \eqref{a12} imply 
		\begin{align*}
			\frac{1}{\varepsilon}\|\v_{\varepsilon}(T)\|_{\H}^2+\int_{Q_{\omega}}|\boldsymbol{\vartheta}_{\varepsilon}|^2\d x \d t=\big(\v_0,\vr(0)\big)-\sum_{i,j=1}^d \int_{Q}\g_{ij}\frac{\partial \vr_{i}}{\partial x_{j}} \d x\d t+\int_{0}^T \langle\h,\vr \rangle \d t.
		\end{align*}
		Using the Cauchy-Schwartz and Young's inequalities, we obtain
		\begin{align}\label{a131}
			&\frac{1}{\varepsilon}\|\v_{\varepsilon}(T)\|_{\H}^2+\int_{Q_{\omega}}|\boldsymbol{\vartheta}_{\varepsilon}|^2\d x \d t \no \\&
			\leq  \frac{3C}{2}\|\v_0\|_{\H}^2+\frac{3C}{2}\int_{Q} e^{-2s \wi{\psi}}|\g|^2 \d x \d t+\frac{3C}{2}\int_{0}^{T} \|e^{-s \wi{\psi}}\h\|_{\mathrm{H}^{-1}(\Omega)^d}^2  \d t \no \\&\quad+ \frac{1}{6C}\|\vr(0)\|_{\H}^2+\frac{1}{3C}\int_{Q}e^{2s \wi{\psi}}| \nabla \vr|^2\d x \d t,
		\end{align}
		where $\wi{\psi}$ and $\wi{\xi}$ are the weight functions defined in \eqref{21} and  $C$ is the constant appearing in the observability inequality (see \eqref{o231}).	From the observability inequality \eqref{o231} and optimality condition \eqref{a13}, we deduce 
		\begin{align}\label{a14}
			&\|\vr(0)\|_{\H}^2+\int_{Q} \wi{\xi}^2 e^{2s \wi{\psi}}| \vr|^2\d x \d t+\int_{Q}e^{2s \wi{\psi}}| \nabla \vr|^2\d x \d t
			\no\\&\leq C   \int_{Q_{\omega}}\wi{\xi}^3e^{2s\wi{\psi}} | \vr|^2\d x \d t \leq C \int_{Q_{\omega}} | \vr|^2\d x \d t=C \int_{Q_{\omega}}| \boldsymbol{\vartheta}_{\varepsilon}|^2 \d x\d t.
		\end{align}
		If we assume that $\g$ and $\h$ satisfy
		\begin{align*}
			\int_{Q} e^{-2s \wi{\psi}}|\g|^2 \d x \d t <+\infty \ \  \text{ and } \ \  \int_{0}^{T} \|e^{-s \wi{\psi}}\h\|_{\mathrm{H}^{-1}(\Omega)^d}^2  \d t<+\infty,
		\end{align*}
		then by using \eqref{a14} in \eqref{a131}, we have the following estimate:
		\begin{align*}
			\frac{1}{\varepsilon}\|\v_{\varepsilon}(T)\|_{\H}^2+\frac{1}{2}\int_{Q_{\omega}}|\boldsymbol{\vartheta}_{\varepsilon}|^2\d x \d t \leq C\bigg(\|\v_0\|_{\H}^2+\int_{Q} e^{-2s \wi{\psi}}|\g|^2 \d x \d t+\int_{0}^{T} \|e^{-s \wi{\psi}}\h\|_{\mathrm{H}^{-1}(\Omega)^d}^2  \d t\bigg).
		\end{align*}
		We can easily see that $\frac{1}{\sqrt{\varepsilon}}\v_{\varepsilon}(T)$ and $\boldsymbol{\vartheta}_{\varepsilon}$ are uniformly bounded in $\H$ and $\mathrm{L}^2(Q_{\omega})^d$, respectively, and independent of $\varepsilon$.
		Using the Banach-Alaoglu theorem, we can extract subsequences, still denoted by		$\boldsymbol{\vartheta}_{\varepsilon}$ and $\v_{\varepsilon}$, such that
		\begin{align*}
			&\boldsymbol{\vartheta}_{\varepsilon} \xrightharpoonup{w}\boldsymbol{\vartheta} \ \ \text{  in } \ \mathrm{L}^2(Q_{\omega})^d,\\
			&	\v_{\varepsilon}=\v(\boldsymbol{\vartheta}_{\varepsilon}) \xrightharpoonup{w} \v=\v(\boldsymbol{\vartheta}) \ \text{ in } \ \C([0,T];\H) \cap \mathrm{L}^2(0,T;\V),\\&
			\v_{\varepsilon}(T) \xrightharpoonup{w} \v(T) \ \text{ in } \ \H.
		\end{align*}
		Using the weakly lowersemicontinuity property of the norm, we have $$\|\v(T)\|_{\H}^2 \leq \liminf_{\varepsilon \rightarrow 0} \|\v_{\varepsilon}(T)\|_{\H}^2.$$
		Thus, we must have
		\begin{align*}
			\v(\cdot,T)=\mathbf{0},
		\end{align*}
		which completes the proof. 
	\end{proof}
	\subsection{Exponentially decreasing controls and solutions} In the following subsection, our aim is to find a control $\boldsymbol{\vartheta},$ which is exponentially decreasing as $t$ approaches $T$, such that not only $\v(\cdot,T)=\mathbf{0}$, but also $\v$ is exponentially decreases as $t$ approaches $T$. This result will be crucial to deduce the controllability properties for the nonlinear system \eqref{a1} in the next section. Motivated by the works \cite{CGIP,Puel}, we have the following result:
	\begin{theorem}\label{thmE}
		If $\v_0 \in \H$ and the functions $\g$ and $\h$ satisfy
		\begin{align*}
			\int_{Q}e^{-2s\wi{\psi}}|\g|^2 \d x\d t <+\infty \ \ \  \text{and} \ \ \ \int_{0}^{T} \|e^{-s \wi{\psi}}\h\|_{\mathrm{H}^{-1}(\Omega)^d}^2  \d t <+\infty,
		\end{align*}
		respectively,	then there exists a control ${\boldsymbol{\vartheta}} \in \mathrm{L}^2(Q_{\omega})^d$ and a solution ${\v}$ of \eqref{a5} such that
		$${\v}(\cdot,T)=\mathbf{0},$$ and
		\begin{align*}
			\int_{Q} e^{-2 s\wi{\psi}}|{\v}|^2 \d x \d t< +\infty \ \ \text{and} \ \ \  \int_{Q_{\omega}}\frac{e^{-2 s\wi{\psi}}}{\xi^3}|{\boldsymbol{\vartheta}}|^2 \d x \d t < + \infty.
		\end{align*}
	\end{theorem}
	\begin{proof}
		Let us remark that we adapt the ideas from \cite{CGIP,SG,Puel} to prove our result.
		Let us define
		\begin{align*}
			\mathcal{X}=\bigg\{ (\bar{\v},\bar{q}) \in \C^{\infty}(\overline{Q})^{d+1}, \  \ \nabla \cdot \bar{\v}=0, \ \text{in} \ Q, \ \bar{\v}=\mathbf{0}, \ \text{on} \ \Sigma, \ \int_{\omega}\bar{q}(t) \d x =0, \  \text{ a.e.} \ \text{in} \ (0,T)\bigg\}.
		\end{align*}
		Let us consider the  bilinear form $\mathfrak{B}: \mathcal{X} \times \mathcal{X} \rightarrow \R$ given by
		\begin{align}\label{B}
			\mathfrak{B}	\big((\breve{\v},\breve{q}),(\bar{\v},\bar{q})\big)&=\int_{Q}e^{2 s\wi{\psi}}(\mathcal{L}^{*}\breve{\v}+\nabla \breve{q})\cdot(\mathcal{L}^{*}\bar{\v}+\nabla \bar{q}) \d x\d t\no\\&\quad+\int_{Q_{\omega}}e^{2 s\wi{\psi}} \wi{\xi}^3\breve{\v} \cdot \bar{\v} \d x\d t, \ \ \text{for all} \  (\bar{\v},\bar{q}) \in \mathcal{X},
		\end{align}
		where $\mathcal{L}^*$ is the adjoint operator of $\mathcal{L}$ (cf. \eqref{a121} and \eqref{a31}).
		
		Since $\wi{\u} \in \mathrm{L}^\infty(Q)^d$, because of the Carleman estimate \eqref{o19}, we can see that $\mathfrak{B}(\cdot,\cdot)$ is a scalar product on $\mathcal{X}$. Let us define $\overline{\mathcal{X}}$ as the completion of $\mathcal{X}$ for the norm associated to scalar product $\mathfrak{B}(\cdot,\cdot)$ (which we denote by $\| \cdot \|_{\overline{\mathcal{X}}}$). Then, $\overline{\mathcal{X}}$ is a Hilbert space for this scalar product. Thus, the bilinear form  $\mathfrak{B}(\cdot,\cdot)$ is well-defined, continuous, and coercive on $\overline{\mathcal{X}}$, and observe that the observability inequality \eqref{o231} holds for all $(\bar{\v},\bar{q}) \in \overline{\mathcal{X}}$, that is,
		\begin{align}\label{o27}
			\|\bar{\v}(0)\|_{\H}^2+\int_{Q}e^{2s\wi{\psi}}\big(|\nabla \bar{\v}|^2+\wi{\xi}^2 |\bar{\v}|^2 \big)\d x \d t
			&\leq C\int_{Q_{\omega}}e^{2 s\wi{\psi}} \wi{\xi}^3|\bar{\v}|^2 \d x\d t \no\\&
			 \leq C \mathfrak{B}\big((\bar{\v},\bar{q}),(\bar{\v},\bar{q})\big), \ \   \text{for all} \ (\bar{\v},\bar{q}) \in \overline{\mathcal{X}}.
		\end{align} 
		Let us now consider the  linear form $\mathfrak{L}:\overline{\mathcal{X}} \rightarrow\R$ defined by
		\begin{align}\label{lf}
			\langle	\mathfrak{L},(\bar{\v},\bar{q}) \rangle =(\v_{0},\bar{\v}(0))_{\H}-\sum_{i,j=1}^d \int_{Q}\g_{ij}\frac{\partial\bar{\v}_{i}}{ \partial x_{j}} \d x\d t+\int_{0}^T\langle\h , \bar{\v}\rangle \d t, \ \ \text{for all} \ (\bar{\v},\bar{q}) \in \overline{\mathcal{X}},
		\end{align}
		with
		\begin{align}\label{L}
			|\langle	\mathfrak{L},(\bar{\v},\bar{q}) \rangle|& \leq \|\v_0\|_{\H}\|\bar{\v}(0)\|_{\H}+\big\|e^{-s \wi{\psi}} \g \big\|_{\mathrm{L}^2(Q)^{d^{2}}} \big\|e^{s \wi{\psi}} \bar{\v} \big\|_{\mathrm{L}^2(0,T;\V)}\no\\&\quad+\big\|e^{-s \wi{\psi}}  \h \big\|_{\mathrm{L}^2(0,T;\mathrm{H}^{-1}(\Omega)^d)} \big\|e^{s \wi{\psi}}  \bar{\v} \big\|_{\mathrm{L}^2(0,T;\V)}\no\\& \leq C\left(\|\v_0\|_{\H}+\big\|e^{-s \wi{\psi}} \g \big\|_{\mathrm{L}^2(Q)^{d^{2}}} +\big\|e^{-s \wi{\psi}}  \h \big\|_{\mathrm{L}^2(0,T;\mathrm{H}^{-1}(\Omega)^d)} \right)\|(\bar{\v},\bar{q})\|_{\overline{\mathcal{X}}},
		\end{align}
		for all $(\bar{\v},\bar{q}) \in \overline{\mathcal{X}}$ and we have used the estimate \eqref{o27}. 	Therefore, we can see that the linear form $\langle\mathfrak{L},(\cdot)\rangle$ is well-defined and continuous  on $\overline{\mathcal{X}}$. Consequently, in view of Lax-Milgram Theorem (see  \cite[Chapter 6]{Evans}), there exists a unique solution $(\breve{\v},\breve{q}) \in \overline{\mathcal{X}}$ of the variational  problem
		\begin{align}\label{vr1}
			\mathfrak{B}	\big((\breve{\v},\breve{q}),(\bar{\v},\bar{q})\big)=\langle	\mathfrak{L},(\bar{\v},\bar{q}) \rangle, \ \ \text{for all} \  (\bar{\v},\bar{q}) \in \overline{\mathcal{X}}. 
		\end{align}
		Let us define
		\begin{equation}\label{sol}
			{\v}=e^{2 s\wi{\psi}}(L^{*}\breve{\v}+\nabla \breve{q}), \ \ \ \text{and} \ \ \ 
			{\boldsymbol{\vartheta}}=-e^{2 s\wi{\psi}} \wi{\xi}^3 \mathds{1}_{\omega}\breve{\v},
		\end{equation}
		and one can  observe that $(\v,\boldsymbol{\vartheta})$ verifies
		\begin{align*}
			\int_{Q} e^{-2 s\wi{\psi}}|{\v}|^2 \d x \d t+	\int_{Q_{\omega}}\frac{e^{-2 s\wi{\psi}}}{\xi^3}|{\boldsymbol{\vartheta}}|^2 \d x \d t  < +\infty
		\end{align*}
		and  $\v$  is solution of the system in \eqref{a5} for certain pressure $q$.
		
		Note that, $	{\v} \in \mathrm{L}^2(Q)^d$ and  $ 	{\boldsymbol{\vartheta}} \in  \mathrm{L}^2(Q_{\omega})^d$. Then, we introduce the weak solution $(\check{\v},\check{q})$ to the system  
		\begin{equation}\label{cv2}
			\left\{
			\begin{aligned}
				\mathcal{L}\check{\v}+\nabla \check{q}&=\nabla \cdot \g+\h+	{\boldsymbol{\vartheta}} \mathds{1}_{\omega}, \  \text{ in } Q, \\ \nabla\cdot\check{\v}&=0,  \ \text{ in } Q,\\
				\check{\v}&=\mathbf{0},  \ \text{ on }  \Sigma,\\
				\check{\v}(0)&=\v_0,  \text{ in }  \Omega.
			\end{aligned}
			\right.
		\end{equation}
		Indeed, $\check{\v}$ is also the unique solution of \eqref{cv2} defined by transposition. Of course, this says that $\check{\v} \in \mathrm{L}^2(Q)^d$ is the unique function verifying
		\begin{align}\label{o30}
			\int_{Q}\check{\v} \cdot \bar{\boldsymbol{b}} \d x\d t&=(\v_{0},\bar{\v}(0))_{\H}-\sum_{i,j=1}^d \int_{Q}\g_{ij}\frac{\partial\bar{\v}_{i}}{ \partial x_{j}} \d x\d t\no \\&\quad+\int_{0}^T \langle\h, \bar{\v} \rangle \d t+\int_{Q_{\omega}}{\boldsymbol{\vartheta}} \cdot \bar{\v} \d x\d t, \ \ \text{ for all } \ \bar{\boldsymbol{b}} \in \mathrm{L}^2(Q)^d.
		\end{align}
		Here $\bar{\v}$, together with some pressure $\bar{q}$, is the solution of the following system
		\begin{equation*}
			\left\{
			\begin{aligned}
				\mathcal{L}\bar{\v}+\nabla \bar{q}&=\bar{\boldsymbol{b}}, \  \text{ in } Q, \\ \nabla\cdot\bar{\v}&=0,  \ \text{ in } Q,\\
				\bar{\v}&=\mathbf{0},  \ \text{ on }  \Sigma,\\
				\bar{\v}(T)&=\mathbf{0},  \text{ in }  \Omega.
			\end{aligned}
			\right.
		\end{equation*}
		From the relations \eqref{sol} and \eqref{vr1}, one can easily get that $\v$ also verifies \eqref{o30}. Consequently, $\v=\check{\v}$ and $\v$ is, together with a certain pressure $q=\check{q}$,  the solution of the system \eqref{a5}.
		
		On the other hand, by virtue of \eqref{B} and \eqref{L}-\eqref{sol}, it follows
		\begin{align*}
			&\int_{Q} e^{-2 s\wi{\psi}}|{\v}|^2 \d x \d t+	\int_{Q_{\omega}}\frac{e^{-2 s\wi{\psi}}}{\xi^3}|{\boldsymbol{\vartheta}}|^2 \d x \d t \\&
			=\int_{Q}e^{2s\wi{\psi}}|\mathcal{L}^{*}\breve{\v}+\nabla \breve{q}|^2 \d x\d t+\int_{Q_{\omega}}e^{2 s\wi{\psi}} \wi{\xi}^3 |\breve{\v}|^2 \d x \d t
			\\& \leq  C\left(\|\v_0\|_{\H}+\big\|e^{-s \wi{\psi}} \g \big\|_{\mathrm{L}^2(Q)^{d^{2}}} +\big\|e^{-s \wi{\psi}}  \h \big\|_{\mathrm{L}^2(0,T;\mathrm{H}^{-1}(\Omega)^d)} \right).
		\end{align*}
		Obviously, this shows that $\v(\cdot,T)=\mathbf{0}$ and that  ${\v}$ and ${\boldsymbol{\vartheta}}$ decrease exponentially as $t$ approaches $T$ and it completes the proof.
	\end{proof}

	\section{The Nonlinear Problem}\label{sec5}\setcounter{equation}{0}
	In this section, we provide the proof of our main result, Theorem \ref{thmm}, using similar arguments to those employed in \cite{CGIP,Puel}, and references therein. Moreover, we see that the results obtained in the previous section allow us to locally invert a nonlinear system. We deduce some regularity results, which will be sufficient for us to apply a suitable inverse mapping theorem (see Theorem \ref{thmnl}) to reach our main goal. 
	
	In the context of the local exact controllability to the trajectories of the CBF equations, we adapt some ideas of the proofs from the works \cite{CGIP,Puel}, where the local exact controllability to the trajectories is established for the Navier-Stokes system in both two and three dimensions.
	\subsection{Choice of special weights} 
	From now on, we omit the notation $ \wi \  $ for the weight functions $\wi{\psi}$ and $\wi{\xi}$ defined in \eqref{21}. We still have some choice for the constants $m_1$ and $m_2$ in the weights ${\psi}$ and ${\xi}$. To do this, first, for $ t \in [T/2,T]$ (the case $t \in [0,T/2]$ is straightforward), we define 
	\begin{align*}
		\check{\psi}(t)&=\min_{x \in \overline{\Omega}}\psi(x,t)=\frac{e^{\lambda m_1}-e^{\lambda (\|\eta\|_{\mathrm{L}^\infty}+m_{2})}}{\ell^{4}(t)}  <0,\\
		\check{\xi}(t)&=\min_{x \in \overline{\Omega}}\xi(x,t)=\frac{e^{\lambda m_1}}{\ell^{4}(t)},\\
		\hat{\psi}(t)&=\max_{x \in \overline{\Omega}}\psi(x,t)=\frac{e^{\lambda (\|\eta\|_{\mathrm{L}^\infty}+m_{1})}-e^{\lambda (\|\eta\|_{\mathrm{L}^\infty}+m_{2})}}{\ell^{4}(t)}  <0,\\
		\hat{\xi}(t)&=\max_{x \in \overline{\Omega}}\xi(x,t)=\frac{e^{\lambda (\|\eta\|_{\mathrm{L}^\infty}+m_{1})}}{\ell^{4}(t)}.
	\end{align*}
	Then, we have the following result for choosing the constants given in \cite{Puel}.
	\begin{lemma}[Lemma 4.1, \cite{Puel}]\label{lemmacc}
		We can choose the constants $m_1$ and $m_2$ with $m_1 \leq m_2$ such that for $\lambda_0$ large enough, we have for $\lambda \geq \lambda_0$
		\begin{align*}
			\bigg|\frac{\partial \psi}{\partial t}\bigg| \leq C \xi^\frac{5}{4}, \ \ \ \text{and} \ \ \ \bigg|\frac{\partial^{2} \psi}{\partial t^2}\bigg| \leq C \xi^\frac{3}{2},
		\end{align*}
		and \begin{align*}
			\frac{3}{2} \hat{\psi} \leq \check{\psi} \ \ \ \text{or} \ \ \ -\check{\psi} \leq -\frac{3}{2} \hat{\psi}.
		\end{align*}
	\end{lemma}
	\begin{proof} Proof follows in the same lines as in the proof of  \cite[Lemma 4.1]{Puel}.
		Let us take 
		\begin{align*}
			m_1=(m_0+4)\| \eta\|_{\mathrm{L}^\infty} \ \ \ \text{and} \ \ \ m_2=m_3\| \eta\|_{\mathrm{L}^\infty}.
		\end{align*}
		It is easy to see that all the requirements are fulfilled if
		\begin{align*}
			m_0+4 < m_3 <\frac{5}{4}m_0+4,
		\end{align*}
		for $\lambda_0$ sufficiently large. Now, such a choice of $m_0$ and $m_3$ is  possible.
	\end{proof}
	\subsection{Functional class and some regularity results}\label{sub5.2}
	To apply an inverse mapping argument and extract some controllability properties for the nonlinear system \eqref{a1}, suitable spaces  are needed. We must define the correct functional class to apply the inverse mapping theorem. To do this, let us define the weighted Banach spaces $(\mathcal{E}_{d},\| \cdot \|_{\mathcal{E}_d}), \ (d=2,3)$ as follows:
	\begin{align*}
		\mathcal{E}_{2}&=\bigg\{\big(\v,\boldsymbol{\vartheta} \big): \big(\v,\boldsymbol{\vartheta} \big) \in \mathcal{E}_0, e^{-  \frac{3s}{4}\hat{\psi}} \v \in  \mathrm{L}^6(0,T;\mathrm{L}^4(\Omega)^2), \\&\qquad
		\exists \ q, \bar{\g}, \bar{\h}:  e^{- s{\psi}} ( \nabla \cdot\bar{\g}+\bar{\h}) \in \mathrm{L}^2(0,T;\mathrm{H}^{-1}(\Omega)^{2}),  \\&\qquad \mathcal{L}\v+\nabla q-\boldsymbol{\vartheta} \mathds{1}_{\omega}=\nabla \cdot \bar{\g}+\bar{\h}, \ \ \v(0) \in \H \bigg\},
	\end{align*}
	when $d=2$ and 
	\begin{align*}
		\mathcal{E}_{3}&=\bigg\{\big(\v,\boldsymbol{\vartheta} \big): \big(\v,\boldsymbol{\vartheta} \big) \in \mathcal{E}_0,   e^{-  \frac{3s}{4}\hat{\psi}} \v \in  \mathrm{L}^4(0,T;\mathrm{L}^{12}(\Omega)^3) \cap \mathrm{L}^6(Q)^3,  \\&\qquad
		\exists \ q,\bar{\g}, \bar{\h}: e^{- s{\psi}} (\nabla \cdot\bar{\g}+\bar{\h}) \in \mathrm{L}^2(0,T;\mathrm{W}^{-1,6}(\Omega)^{3}),  \\&\qquad \mathcal{L}\v+\nabla q-\boldsymbol{\vartheta} \mathds{1}_{\omega}=\nabla \cdot \bar{\g}+\bar{\h}, \ \ \v(0) \in \H \cap \mathrm{L}^4(\Omega)^3 \bigg\},
	\end{align*}
	when $d=3$,	where
	\begin{align*}
		\mathcal{E}_0&=\bigg\{\big(\v,\boldsymbol{\vartheta} \big):e^{- s{\psi}} \v \in \mathrm{L}^2(Q)^d, \ \ \xi^{-\frac{3}{2}}e^{- s{\psi}}\boldsymbol{\vartheta} \in \mathrm{L}^2(Q_{\omega})^d, \\&\qquad  e^{-  \frac{3s}{4}\hat{\psi}} \v \in  \mathrm{L}^\infty(0,T;\H) \cap\mathrm{L}^2(0,T;\V) \bigg\},
	\end{align*}
	endowed with the norm
	\begin{align*}
		\|(\v,\boldsymbol{\vartheta})\|_{\mathcal{E}_2}^2&=\left\|e^{- s{\psi}} \v \right\|_{\mathrm{L}^2(Q)^2}^2+\big\|\xi^{-\frac{3}{2}}e^{- s{\psi}}\boldsymbol{\vartheta}\big\|_{\mathrm{L}^2(Q_{\omega})^2}^2 +\big\| e^{-  \frac{3s}{4}\hat{\psi}} \v \big\|_{\mathrm{L}^6(0,T;\mathrm{L}^4(\Omega)^2)}^2 \\&\quad +\big\|e^{-  \frac{3s}{4}\hat{\psi}} \v \big\|_{\mathrm{L}^\infty(0,T;\H) \cap\mathrm{L}^2(0,T;\V)}^2 +\left\| e^{- s{\psi}} \nabla \cdot\bar{\g} \right\|_{\mathrm{L}^2(0,T;\mathrm{H}^{-1}(\Omega)^2)}^2 \\&\quad+\left\|e^{- s{\psi}} \bar{\h} \right\|_{\mathrm{L}^2(0,T;\mathrm{H}^{-1}(\Omega)^{2})}^2+\|\v(0)\|_{\H}^2,
	\end{align*}
	when $d=2$ and 
	\begin{align*}
		\|(\v,\boldsymbol{\vartheta})\|_{\mathcal{E}_3}^2&=\left\|e^{- s{\psi}} \v \right\|_{\mathrm{L}^2(Q)^{3}}^2+\big\|\xi^{-\frac{3}{2}}e^{- s{\psi}}\boldsymbol{\vartheta}\big\|_{\mathrm{L}^2(Q_{\omega})^3}^2 
		+\big\|e^{-  \frac{3s}{4}\hat{\psi}} \v \big\|_{\mathrm{L}^\infty(0,T;\H) \cap\mathrm{L}^2(0,T;\V)}^2
		\\&\quad +\big\| e^{-  \frac{3s}{4}\hat{\psi}} \v \big\|_{\mathrm{L}^6(Q)^3}^2 +\big\|e^{- \frac{3s}{4}\hat{\psi}} \v \big\|_{\mathrm{L}^4(0,T;\mathrm{L}^{12}(\Omega)^3)}^2 \\& \quad +\left\| e^{- s{\psi}} \nabla \cdot \bar{\g} \right\|_{\mathrm{L}^2(0,T;\mathrm{W}^{-1,6}(\Omega)^3)}^2+\left\|e^{- s{\psi}} \bar{\h} \right\|_{\mathrm{L}^2(0,T;\mathrm{W}^{-1,6}(\Omega)^{3})}^2+\|\v(0)\|_{\mathrm{L}^4(\Omega)^3}^2,
	\end{align*}
	when $d=3$. Here, the operator $\mathcal{L}\v$ is defined in \eqref{a31}. On the other hand, we also define the weighted Banach spaces $(\mathcal{G}_{d},\| \cdot \|_{\mathcal{G}_d})$ for  $d=2,3$ as:
	\begin{align*}
		\mathcal{G}_2= \bigg\{ \big(\nabla \cdot \bar{\g}+\bar{\h},\v_0\big): \  e^{- s{\psi}} (  \nabla \cdot\bar{\g}+\bar{\h}) \in \mathrm{L}^2(0,T;\mathrm{H}^{-1}(\Omega)^{2}), \  \ \v_0 \in \H\bigg\},
	\end{align*}
	and 
	\begin{align*}
		\mathcal{G}_3=\bigg\{	\big(\nabla \cdot \bar{\g}+\bar{\h},\v_0\big): \ e^{- s{\psi}} (\nabla \cdot\bar{\g}+\bar{\h}) \in \mathrm{L}^2(0,T;\mathrm{W}^{-1,6}(\Omega)^{3}), \ \ \v_0 \in \H \cap \mathrm{L}^4(\Omega)^3\bigg\},
	\end{align*}
	endowed with the norm
	\begin{align*}
		\|(\nabla \cdot \bar{\g}+\bar{\h},\v_0)\|_{\mathcal{G}_2}^2=\left\|e^{- s{\psi}} \nabla \cdot \bar{\g} \right\|_{\mathrm{L}^2(0,T;\mathrm{H}^{-1}(\Omega)^{2})}^2+\| e^{- s{\psi}} \bar{\h}\|_{\mathrm{L}^2(0,T;\mathrm{H}^{-1}(\Omega)^2)}^2+\|\v_0\|_{\H}^2,
	\end{align*}
	and
	\begin{align*}
		\|(\nabla \cdot \bar{\g}+\bar{\h},\v_0)\|_{\mathcal{G}_3}^2=\left\|e^{- s{\psi}} \nabla \cdot \bar{\g} \right\|_{\mathrm{L}^2(0,T;\mathrm{W}^{-1,6}(\Omega)^{3})}^2+\| e^{- s{\psi}} \bar{\h}\|_{\mathrm{L}^2(0,T;\mathrm{W}^{-1,6}(\Omega)^3)}^2+\|\v_0\|_{\mathrm{L}^4(\Omega)^3}^2.
	\end{align*}

	\begin{remark}
		The spaces denoted as $\mathcal{E}_j \ (j=0,2,3)$ represent the natural spaces, where solutions of the null controllability of \eqref{a5} must be found in order to preserve these properties for the nonlinear terms $(\u \cdot \nabla)\u$ and $|\u|^2\u$. Additional information will be provided  in subsection \ref{subsec 5.3}.
	\end{remark}
	
	Given any $(\nabla \cdot \g^*+\h^*, \v_0) \in \mathcal{G}$, we know from Theorem \ref{thmE} that $(\v,\boldsymbol{\vartheta}^*)$ solves the system:
	\begin{equation}\label{cv3}
		\left\{
		\begin{aligned}
			\mathcal{L}\v+\nabla q&=\nabla \cdot \g^*+\h^*+\boldsymbol{\vartheta}^* \mathds{1}_{\omega}, \ \text{ in } Q, \\ \nabla\cdot\v&=0,  \ \text{ in } Q,\\
			\v&=\mathbf{0},  \ \text{ on }  \Sigma,\\
			\v(0)&=\v_0, \text{ in }  \Omega,
		\end{aligned}
		\right.
	\end{equation}
	with $$\v(\cdot,T)=\mathbf{0},$$ and 
	\begin{align*}
		\v \in  \mathrm{L}^\infty(0,T;;\H) \cap \mathrm{L}^2(0,T;\V), \ \ \ \ 
		e^{-s\psi}\v \in \mathrm{L}^2(Q)^d,  \ \ \ \text{and} \ \ \ \xi^{-\frac{3}{2}}{e^{-s\psi}} \boldsymbol{\vartheta}^* \in \mathrm{L}^2(Q_{\omega})^{d}.
	\end{align*}
Let us prove that $(\v,\boldsymbol{\vartheta}^*) \in \mathcal{E}$.	It remains to prove that
	\begin{itemize}
		\item  $e^{-\frac{3s}{4}\hat{\psi}}\v \in   \mathrm{L}^\infty(0,T;\H) \cap\mathrm{L}^2(0,T;\V),$ \ \ in dimension $d=2,3$,
		\item  $e^{-\frac{3s}{4}\hat{\psi}}\v \in  \mathrm{L}^4(0,T;\mathrm{L}^{12}(\Omega)^3) \cap \mathrm{L}^6(Q)^3$, \ \ in dimension $d=3$,
		\item  $e^{-\frac{3s}{4}\hat{\psi}}\v \in  \mathrm{L}^6(0,T;\mathrm{L}^{4}(\Omega)^2),$ \ \ in dimension $d=2$.
	\end{itemize}
	Let us first deduce that $e^{-\frac{3s}{4}\hat{\psi}}\v \in   \mathrm{L}^\infty(0,T;\H) \cap\mathrm{L}^2(0,T;\V).$ To fulfill this purpose, let us define the functions
	\begin{align*}
		\hat{\v}=	e^{-\frac{3s}{4}\hat{\psi}}\v, \ \ 	\hat{q}=	e^{-\frac{3s}{4}\hat{\psi}} q, \ \ 	\hat{\g}=	e^{-\frac{3s}{4}\hat{\psi}}\g^*, \ \ 	\hat{\h}=	e^{-\frac{3s}{4}\hat{\psi}}\h^*, \ \ \text{and} \ \ 	\hat{\boldsymbol{\vartheta}}=	e^{-\frac{3s}{4}\hat{\psi}}\boldsymbol{\vartheta}^*.
	\end{align*}
	Then, the function $(\hat{\v},\hat{q})$ solves the following system:
	\begin{equation}\label{cv11}
		\left\{
		\begin{aligned}
			\mathcal{L}\hat{\v}+\nabla \hat{q}&=\nabla \cdot \hat{\g}+\hat{\h}+\hat{\boldsymbol{\vartheta}} \mathds{1}_{\omega}-\frac{3s}{4}\frac{\partial \hat{\psi}}{\partial t} \hat{\v}, \ \text{ in } Q, \\ \nabla\cdot\hat{\v}&=0,  \ \text{ in } Q,\\
			\hat{\v}&=\mathbf{0}, \ \text{ on }  \Sigma,\\
			\hat{\v}(0)&=e^{-\frac{3s}{4}\hat{\psi}(0)}\v_0, \ \text{ in }  \Omega.
		\end{aligned}
		\right.
	\end{equation}
	It is easy to check that
	\begin{align*}
		\nabla \cdot\hat{\g} ,\  \hat{\h} \in  \mathrm{L}^2(0,T;\mathrm{H}^{-1}(\Omega)^{d}), \  \text{and} \ \hat{\boldsymbol{\vartheta}} \in \mathrm{L}^2(Q_{\omega})^d,
	\end{align*}
	and
	\begin{align*}
		\frac{\partial \hat{\psi}}{\partial t} \hat{\v}=\frac{\partial \hat{\psi}}{\partial t} e^{-\frac{3s}{4}\hat{\psi}}\v=\bigg(\frac{\partial \hat{\psi}}{\partial t} e^{\frac{s}{4}\hat{\psi}}\bigg)e^{-s\hat{\psi}}\v \in  \mathrm{L}^2(Q)^d.
	\end{align*}
	Therefore, for given $\v_0 \in \H$, from the existence result for $d=2,3$, we have $$\hat{\v} \in \mathrm{L}^\infty(0,T;\H) \cap \mathrm{L}^2(0,T;\V).$$
	
	Now, we will prove that $\hat{\v} \in \mathrm{L}^4(0,T;\mathrm{L}^{12}(\Omega)^3) \cap \mathrm{L}^6(Q)^3$ when $d=3$.
	Since $\wi{\u} \in \mathrm{L}^\infty(Q)^3$ and $\hat{\v} \in \mathrm{L}^\infty(0,T;\H) \cap \mathrm{L}^2(0,T;\V)$, we observe that
	\begin{align*}
		\alpha\hat{\v}+\beta|\wi{\u}|^{2}\hat{\v}+2 \beta (\hat{\v} \cdot \wi{\u}) \wi{\u}  \in  \mathrm{L}^2(Q)^3.
	\end{align*}
	Also, $\hat{\v} \in  \mathrm{L}^2(0,T;\V) \subset \mathrm{L}^2(0,T;\mathrm{L}^6(\Omega)^3)$, we obtain
	\begin{align*}
		\nabla \cdot	(\hat{\v} \otimes\wi{\u}), \nabla \cdot (\wi{\u} \otimes \hat{\v}) \in \mathrm{L}^2(0,T;\mathrm{W}^{-1,6}(\Omega)^{3}), 
	\end{align*}
	and these terms can be combined in $\nabla \cdot \hat{\g}$, so that without loss of generality, we can write (with the same notations as in \eqref{cv11})
	\begin{equation}\label{cv12}
		\left\{
		\begin{aligned}
			\frac{\partial\hat{\v}}{\partial t}-\mu \Delta \hat{\v}+\nabla \hat{q}&=\nabla \cdot \hat{\g}+\hat{\h}+\hat{\k}, \ \text{ in } Q, \\ \nabla\cdot\hat{\v}&=0, \ \text{ in } Q,\\
			\hat{\v}&=\mathbf{0},  \ \text{ on }  \Sigma,\\
			\hat{\v}(0)&=e^{-\frac{3s}{4}\hat{\psi}(0)}\v_0, \ \text{ in }  \Omega,
		\end{aligned}
		\right.
	\end{equation}
	where \begin{align*}
		& \ \ \ \  \ \ \ \ \nabla \cdot\hat{\g} , \ \hat{\h} \in  \mathrm{L}^2(0,T;\mathrm{W}^{-1,6}(\Omega)^{3}),
		\end{align*}
and 
\begin{align*}
		 \hat{\k}=\hat{\boldsymbol{\vartheta}} \mathds{1}_{\omega}	-\alpha\hat{\v}-\beta|\wi{\u}|^{2}\hat{\v}-2 \beta (\hat{\v} \cdot \wi{\u}) \wi{\u} -\frac{3s}{4}\frac{\partial \hat{\psi}}{\partial t} \hat{\v} \in \mathrm{L}^2(Q)^3.
	\end{align*}
	Let us consider the Stokes system:
	\begin{equation}\label{cv4}
		\left\{
		\begin{aligned}
			-\frac{\partial{\v}}{\partial t}-\mu \Delta {\v}+\nabla q&={\k}_3, \ \text{ in } Q, \\ \nabla\cdot{\v}&=0, \ \text{ in } Q,\\
			{\v}&=\mathbf{0},  \ \text{ on }  \Sigma,\\
			{\v}(T)&=\mathbf{0},  \ \text{ in }  \Omega.
		\end{aligned}
		\right.
	\end{equation}
	\begin{remark}\label{remGS}
		Based on the Giga-Sohr regularity result (see  \cite[Theorem 2.8]{GS} and  Solonnikov \cite{VAS} also), for  every  ${\k}_3 \in \mathrm{L}^{r_1}(0,T;\mathrm{L}^{r_2}(\Omega)^d)$ and  $1 <r_1, r_2 <+\infty$, there
		exists a unique solution $(\v,\nabla q)$  to the Stokes system \eqref{cv4} satisfying
		\begin{align*}
			&	\v \in \mathrm{L}^{r_1}(0,T_0;\mathrm{W}^{2,r_2}(\Omega)^d)  \ \ \  \text{for all} \ \ \ 0<T_0 \leq T \ \ \  \text{with} \ \ \ T_0 <+\infty,   \\ &   \ \ \ \ \ \ \ \ \ \ \ \  \frac{\partial \v}{\partial t}, \ \nabla q \in \mathrm{L}^{r_1}(0,T;\mathrm{L}^{r_2}(\Omega)^d).
			\end{align*}
			Moreover, we have
			\begin{align*}
				\int_0^T \bigg\|\frac{\partial \v}{\partial t}\bigg\|_{r_2}^{r_1}\d t+\int_0^T \| \Delta \v\|_{r_2}^{r_1}\d t+\int_0^T \|\nabla q\|_{r_2}^{r_1} \d t\leq C \int_0^T \|\k_3\|_{r_2}^{r_1}\d t,
			\end{align*}
			where $C=C(r_1,r_2,\Omega) $ is a positive constant.
	\end{remark}
	
	\begin{lemma}\label{lemr1}
		Let us assume that $d=3$,  $\v_0 \in \H \cap \mathrm{L}^4(\Omega)^3$, $\wi{\u} \in \mathrm{L}^\infty(Q)^3$, and that $\nabla \cdot	\hat{\g}, \ \hat{\h} \in  \mathrm{L}^2(0,T;\mathrm{W}^{-1,6}(\Omega)^{3}),$ and $\hat{\k} \in \mathrm{L}^2(Q)^3$. Then, $\hat{\v} \in \mathrm{L}^4(0,T;\mathrm{L}^{12}(\Omega)^3)$.
	\end{lemma}
	This lemma is a consequence of the following one by the transposition method.
	\begin{lemma}\label{lemr2}
		Let $d=3$. Then, for each ${\k}_3 \in \mathrm{L}^\frac{4}{3}(0,T;\mathrm{L}^\frac{12}{11}(\Omega)^3)$,  there exists a unique solution $(\v,q)$ to the Stokes system \eqref{cv4} satisfying
		\begin{align*}
			\v \in \C([0,T];\mathrm{L}^\frac{4}{3}(\Omega)^3) \cap \mathrm{L}^2(0,T;\mathrm{W}_{0}^{1,\frac{6}{5}}(\Omega)^3).
		\end{align*}
	\end{lemma}
	\begin{proof}
		Let us first show that $\v \in \C([0,T];\mathrm{L}^\frac{4}{3}(\Omega)^3)$. From the Giga-Sohr regularity result on the Stokes problem (see Remark \ref{remGS}), we have
		\begin{align*}
			\v \in \mathrm{L}^\frac{4}{3}(0,T;\mathrm{W}^{2,\frac{12}{11}}(\Omega)^3)  \  \text{ and } \  \frac{\partial \v}{\partial t} \in \mathrm{L}^\frac{4}{3}(0,T;\mathrm{L}^\frac{12}{11}(\Omega)^3).
		\end{align*}
		From the Sobolev embedding theorem, we obtain
		\begin{align*}
			\v \in \mathrm{L}^\frac{4}{3}(0,T;	\mathrm{W}^{2,\frac{12}{11}}(\Omega)^3) \hookrightarrow \mathrm{L}^\frac{4}{3}(0,T;\mathrm{W}^{1,\frac{12}{7}}(\Omega)^3) \hookrightarrow \mathrm{L}^\frac{4}{3}(0,T;\mathrm{L}^4(\Omega)^3).
		\end{align*}
		By interpolation arguments (see \cite{JB} and \cite{LT}), we know that, if 
		\begin{align*}
			\v \in  \mathrm{L}^\frac{4}{3}(0,T;\mathrm{L}^4(\Omega)^3), \ \ \ \text{and} \ \ \ \frac{\partial \v}{\partial t} \in \mathrm{L}^\frac{4}{3}(0,T;\mathrm{L}^\frac{12}{11}(\Omega)^3),
		\end{align*}
		then
		\begin{align*}
			\v \in \C\big([0,T];\big(\mathrm{L}^{4}(\Omega)^3,\mathrm{L}^\frac{12}{11}(\Omega)^3\big)_{\frac{3}{4},\frac{4}{3}}\big).
		\end{align*}
		The interpolation space $\big(\mathrm{L}^{4}(\Omega)^3,\mathrm{L}^\frac{12}{11}(\Omega)^3\big)_{\frac{3}{4},\frac{4}{3}}$ is the Lorentz space $\mathrm{L}^{\frac{4}{3},\frac{4}{3}}(\Omega)^3=\mathrm{L}^{\frac{4}{3}}(\Omega)^3$ (see \cite[Theorem 5.2.1]{JB}). Consequently, we have $\v \in \C([0,T];\mathrm{L}^{\frac{4}{3}}(\Omega)^3)$.
		
		On the other hand, we have
		\begin{align}\label{5.5}
			\v \in \mathrm{L}^\frac{4}{3}\big(0,T;\mathrm{W}^{2,\frac{12}{11}}(\Omega)^3\cap \mathrm{W}^{1,\frac{12}{7}}(\Omega)^3\big) \cap \mathrm{L}^\infty(0,T;\mathrm{L}^\frac{12}{11}(\Omega)^3).
		\end{align}
		We can use interpolation results in \eqref{5.5} to obtain
		\begin{align*}
			\v \in \mathrm{L}^2\big(0,T;(\mathrm{W}^{2,\frac{12}{11}}(\Omega)^3,\mathrm{L}^\frac{12}{11}(\Omega)^3)_{\frac{1}{3},\frac{12}{11}}\big),
		\end{align*}
		and by the definition of Sobolev spaces (see \cite[Chapter 34]{LT}), we have
		\begin{align*}
			\big(\mathrm{W}^{2,\frac{12}{11}}(\Omega)^3,\mathrm{L}^\frac{12}{11}(\Omega)^3\big)_{\frac{1}{3},\frac{12}{11}}=\mathrm{W}^{\frac{4}{3},\frac{12}{11}}(\Omega)^3.
		\end{align*}
		By the virtue of Sobolev's inequality
		$	\mathrm{W}^{\frac{4}{3},\frac{12}{11}}(\Omega)^3  \hookrightarrow \mathrm{W}^{1,\frac{6}{5}}(\Omega)^3$,
		and  $\v=\mathbf{0}$ on the boundary $\Sigma$ in the sense of trace,  we deduce that
		\begin{align*}
			\v \in \mathrm{L}^2\big(0,T;\mathrm{W}_0^{1,\frac{6}{5}}(\Omega)^3\big),
		\end{align*}
		which completes the proof. 
	\end{proof}

	\begin{proof}[Proof of Lemma \ref{lemr1}]
		By definition, we say that  $\hat{\v}$ is the solution by  transposition of the problem for the Stokes equation \eqref{cv12}. This mean that $\hat{\v} \in \mathrm{L}^4(0,T;\mathrm{L}^{12}(\Omega)^3)$, is  the unique solution of \eqref{cv12} satisfying: 
		\begin{align*}
			\int_{Q}\hat{\v}\cdot {\k}_3 \d x \d t&= \int_{\Omega} e^{-\frac{3s}{4}\hat{\psi}(0)}\v_0 \cdot \v(0) \d x+\int_{0}^{T} \langle \nabla \cdot \hat{\g}+\hat{\h} , \v \rangle_{\mathrm{W}^{-1,6}(\Omega)^3, \mathrm{W}^{1,\frac{6}{5}}_{0}(\Omega)^3} \d t\\&\quad+\int_{Q} \hat{\k} \cdot\v \d x \d t, \ \  \text{ for all } \ {\k}_3 \in \mathrm{L}^\frac{4}{3}(0,T;\mathrm{L}^\frac{12}{11}(\Omega)^3),
		\end{align*}
		where $(\v,q)$ is the solution to \eqref{cv4} associated to $\k_3$.	We know that 
		\begin{align*}
			\v(0) \in \mathrm{L}^\frac{4}{3}(\Omega)^3,  \ \ \v \in \mathrm{L}^2(0,T;\mathrm{W}_{0}^{1,\frac{6}{5}}(\Omega)^3) \subset \mathrm{L}^2(Q)^3,
		\end{align*}
		and \begin{align*}
			\v_0 \in \mathrm{L}^{4}(\Omega)^3, \ \ \nabla \cdot \hat{\g}, \ \hat{\h} \in  \mathrm{L}^2(0,T;\mathrm{W}^{-1,6}(\Omega)^{3}), \ \hat{\k} \in \mathrm{L}^2(Q)^3.
		\end{align*}
		Note that, for every ${\k}_3 \in \mathrm{L}^\frac{4}{3}(0,T;\mathrm{L}^\frac{12}{11}(\Omega)^3)$, the Stokes system \eqref{cv4} possesses exactly one solution $\v \in \mathrm{L}^4(0,T;\mathrm{L}^{12}(\Omega)^3)$
		that depends continuously on  ${\k}_3$. Therefore, $\hat{\v}$ is well-defined and defines a unique element $\hat{\v} \in \left( \mathrm{L}^\frac{4}{3}(0,T;\mathrm{L}^\frac{12}{11}(\Omega)^3)\right)'\equiv\mathrm{L}^4(0,T;\mathrm{L}^{12}(\Omega)^3)$, which completes the proof of Lemma \ref{lemr1}.
	\end{proof}

	Now, our task is to deduce that $\hat{\v} \in \mathrm{L}^6(Q)^3$. For this purpose, we establish the following lemma:
	\begin{lemma}\label{lemr4}
		Let $d=3$. Then, for each ${\k}_3 \in \mathrm{L}^\frac{6}{5}(Q)^3$, there exists a unique solution $(\v,q)$ to the Stokes system \eqref{cv4} satisfying
		\begin{align*}
			\v \in \C([0,T];\mathrm{L}^\frac{4}{3}(\Omega)^3) \cap \mathrm{L}^2(0,T;\mathrm{W}_{0}^{1,\frac{6}{5}}(\Omega)^3).
		\end{align*}
	\end{lemma}
	Let us assume that Lemma \ref{lemr4} holds. Then, $\hat{\v}$ must coincide with the solution by transposition of the problem \eqref{cv12}, namely, the unique function $\hat{\v} \in \mathrm{L}^6(Q)^3$ satisfying
	\begin{align*}
		\int_{Q}\hat{\v} \cdot {\k}_3 \d x \d t&= \int_{\Omega}  e^{-\frac{3s}{4}\hat{\psi}(0)}\v_0 \cdot\v(0) \d x+\int_{0}^T\langle\boldsymbol{F}, \v \rangle_{\mathrm{W}^{-1,6}(\Omega)^3,\mathrm{W}_{0}^{1,\frac{6}{5}}(\Omega)^3} \d t,
	\end{align*}
	for all $ {\k}_3 \in \mathrm{L}^\frac{6}{5}(Q)^3$ and $\boldsymbol{F}$ stands for the function
	\begin{align}\label{d1}
		\boldsymbol{F}&= \nabla \cdot \hat{\g}+\hat{\h}+\hat{\boldsymbol{\vartheta}} \mathds{1}_{\omega}-\frac{3s}{4}\frac{\partial \hat{\psi}}{\partial t} \hat{\v}\no\\&\quad-\big((\wi{\u}\cdot \nabla) \hat{\v} +(\hat{\v} \cdot \nabla)\wi{\u}+\alpha\hat{\v}+\beta|\wi{\u}|^{2}\hat{\v}+2 \beta (\hat{\v} \cdot \wi{\u}) \wi{\u}\big),
	\end{align}
	and $(\v,q)$ is the solution of the Stokes system \eqref{cv4} associated to $\k_3$. Remark that, as we already had that $\hat{\v} \in \mathrm{L}^2(0,T;\V) \subset \mathrm{L}^2(0,T;\mathrm{L}^{6}(\Omega)^3)$, all the terms of the previous definition \eqref{d1} make sense by virtue of Lemma \ref{lemr4} and the assumption $\v_0 \in \mathrm{L}^4(\Omega)^3$ (see Remark \ref{remGS}). Therefore, $\hat{\v} \in \left( \mathrm{L}^\frac{6}{5}(Q)^3\right)'\equiv\mathrm{L}^6(Q)^3$.
	
	Let us provide a proof of Lemma \ref{lemr4}, which is based on interpolation arguments.
	\begin{proof}[Proof of Lemma \ref{lemr4}]
		Let us first prove that $\v \in \C([0,T];\mathrm{L}^\frac{4}{3}(\Omega)^3)$. From the Giga-Sohr regularity result for the Stokes problem (see Remark \ref{remGS}), we have that
		\begin{align*}
			\v \in \mathrm{L}^\frac{6}{5}(Q)^3  \ \ \ \text{and} \ \ \ \frac{\partial \v}{\partial t} \in \mathrm{L}^\frac{6}{5}(Q)^3.
		\end{align*}
		The Sobolev embedding yields
		\begin{align*}
			\mathrm{W}^{2,\frac{6}{5}}(\Omega)^3 \hookrightarrow \mathrm{H}^{1}(\Omega)^3 \hookrightarrow \mathrm{L}^6(\Omega)^3.
		\end{align*}
		By the virtue of interpolation results, we have that, if 
		\begin{align*}
			\v \in \mathrm{L}^\frac{6}{5}(0,T;\mathrm{L}^{6}(\Omega)^3) \  \ \ \text{and} \ \ \ \frac{\partial \v}{\partial t} \in \mathrm{L}^\frac{6}{5}(Q)^3,
		\end{align*}
		then
		\begin{align*}
			\v \in \C\big([0,T];\big(\mathrm{L}^6(\Omega)^3,\mathrm{L}^\frac{6}{5}(\Omega)^3\big)_{\frac{7}{8},\frac{4}{3}}\big),
		\end{align*}
		and the interpolation space $\big(\mathrm{L}^6(\Omega)^3,\mathrm{L}^\frac{6}{5}(\Omega)^3\big)_{\frac{7}{8},\frac{4}{3}}$ coincides with the Lorentz space $\mathrm{L}^{\frac{4}{3},\frac{4}{3}}(\Omega)^3=\mathrm{L}^{\frac{4}{3}}(\Omega)^3$ (see  \cite[Theorem 5.2.1]{JB}). Therefore, we get $\v \in \C([0,T];\mathrm{L}^\frac{4}{3}(\Omega)^3)$.
		
		On the other hand, we also have
		\begin{align*}
			\v \in \mathrm{L}^\frac{6}{5}(0,T;\mathrm{W}^{2,\frac{6}{5}}(\Omega)^3 \cap \mathrm{H}_0^{1}(\Omega)^3) \  \ \ \text{and} \ \ \ \v \in \mathrm{L}^\infty(0,T;\mathrm{L}^\frac{6}{5}(\Omega)^3).
		\end{align*}
		Using the interpolation results, we deduce that
		\begin{align*}
			\v \in \mathrm{L}^2\big(0,T;(\mathrm{W}^{2,\frac{6}{5}}(\Omega)^3,\mathrm{L}^{\frac{6}{5}}(\Omega)^3)_{\frac{1}{3},\frac{6}{5}}\big).
		\end{align*}
		But we know that (see \cite[Chapter 34]{LT})
		\begin{align*}
			\big(\mathrm{W}^{2,\frac{6}{5}}(\Omega)^3,\mathrm{L}^{\frac{6}{5}}(\Omega)^3\big)_{\frac{1}{3},\frac{6}{5}}=\mathrm{W}^{\frac{4}{3},\frac{6}{5}}(\Omega)^3.
		\end{align*}
		Thanks to Sobolev's embedding 
		$	\mathrm{W}^{\frac{4}{3},\frac{6}{5}}(\Omega)^3  \hookrightarrow \mathrm{W}^{1,\frac{6}{5}}(\Omega)^3$
		and  $\v(t) \in \mathrm{H}_0^{1}(\Omega)^3,$ a.e. $t \in (0,T)$, we finally infer that
		\begin{align*}
			\v \in \mathrm{L}^2\big(0,T;\mathrm{W}_0^{1,\frac{6}{5}}(\Omega)^3\big),
		\end{align*}
		which completes the proof. 
	\end{proof}
Let us now consider the case $d=2$. 	Our  task is to prove that $\hat{\v} \in \mathrm{L}^6(0,T;\mathrm{L}^4(\Omega)^{2})$. 
	\begin{lemma}\label{lemr5}
		Let $d=2$. Then, for each ${\k}_3 \in \mathrm{L}^\frac{6}{5}(0,T;\mathrm{L}^\frac{4}{3}(\Omega)^2)$, there exists a unique solution $(\v,q)$ to the Stokes system \eqref{cv4} satisfying
		\begin{align*}
			\v \in \C([0,T];\H) \cap \mathrm{L}^2(0,T;\V).
		\end{align*}
	\end{lemma}
	Let us assume that Lemma \ref{lemr5} holds. Then, $\hat{\v}$ must coincide with the solution by transposition of the problem \eqref{cv12}, namely, the unique function $\hat{\v} \in \mathrm{L}^6(0,T;\mathrm{L}^4(\Omega)^{2})$ verifying
	\begin{align*}
		\int_{Q}\hat{\v} \cdot {\k}_3 \d x \d t&= \int_{\Omega}  e^{-\frac{3s}{4}\hat{\psi}(0)}\v_0 \cdot\v(0) \d x+\int_{0}^T\langle\boldsymbol{F^*}, \v \rangle_{\mathrm{H}^{-1}(\Omega)^2,\mathrm{H}_{0}^{1}(\Omega)^2} \d t,
	\end{align*}
	for all $ {\k}_3 \in \mathrm{L}^\frac{6}{5}(0,T;\mathrm{L}^\frac{4}{3}(\Omega)^2)$ (cf. Remark \ref{remGS}). Here, the function  $\boldsymbol{F^*}$ is
	\begin{align}\label{d12}
		\boldsymbol{F^*}&= \nabla \cdot \hat{\g}+\hat{\h}+\hat{\boldsymbol{\vartheta}} \mathds{1}_{\omega}-\frac{3s}{4}\frac{\partial \hat{\psi}}{\partial t} \hat{\v}\no\\&\quad-\big((\wi{\u}\cdot \nabla) \hat{\v} +(\hat{\v} \cdot \nabla)\wi{\u}+\alpha\hat{\v}+\beta|\wi{\u}|^{2}\hat{\v}+2 \beta (\hat{\v} \cdot \wi{\u}) \wi{\u}\big),
	\end{align}
	and $(\v,q)$ is the solution of the Stokes system \eqref{cv4} associated to $\k_3$. By using the similar arguments as Lemma \ref{lemr1} and  by the virtue of Lemma \ref{lemr5}, we finally obtain $\hat{\v} \in \mathrm{L}^6(0,T;\mathrm{L}^{4}(\Omega)^2)$.
	\begin{proof}[Proof of Lemma \ref{lemr5}]
		Let us first deduce that $\v \in \mathrm{L}^2(0,T;\mathrm{H}_0^1(\Omega)^2)$. From the Giga-Sohr regularity result for the Stokes problem (see Remark \ref{remGS}), we have that
		\begin{align*}
			\v \in \mathrm{L}^\frac{6}{5}(0,T;\mathrm{W}^{2,\frac{4}{3}}(\Omega)^2 \cap \mathrm{H}_0^1(\Omega)^2)  \ \ \ \text{and} \ \ \ \frac{\partial \v}{\partial t} \in \mathrm{L}^\frac{6}{5}(0,T;\mathrm{L}^\frac{4}{3}(\Omega)^2).
		\end{align*}
		We also have
		\begin{align*}
			\v \in \mathrm{L}^\frac{6}{5}(0,T;\mathrm{W}^{2,\frac{4}{3}}(\Omega)^2)  \cap \mathrm{L}^\infty(0,T;\mathrm{L}^\frac{4}{3}(\Omega)^2).
		\end{align*}
		The interpolation results gives
		\begin{align*}
			\v \in \mathrm{L}^2\big(0,T;(\mathrm{W}^{2,\frac{4}{3}}(\Omega)^2,\mathrm{L}^{\frac{4}{3}}(\Omega)^2)_{\frac{1}{4},\frac{4}{3}}\big).
		\end{align*}
		But we know that (see \cite[Chapter 34]{LT})
		\begin{align*}
			\big(\mathrm{W}^{2,\frac{4}{3}}(\Omega)^2,\mathrm{L}^{\frac{4}{3}}(\Omega)^3\big)_{\frac{1}{4},\frac{4}{3}}=\mathrm{W}^{\frac{3}{2},\frac{4}{3}}(\Omega)^2.
		\end{align*}
		The Sobolev embedding theorem
		$	\mathrm{W}^{\frac{3}{2},\frac{4}{3}}(\Omega)^2 \hookrightarrow \mathrm{H}^{1}(\Omega)^2$ and 
		$\v=\mathbf{0}$ on the boundary $\Sigma$ in the sense of trace results to
		$$\v \in \mathrm{L}^2\big(0,T;\mathrm{H}_0^{1}(\Omega)^2\big).$$
		
		On the other hand, 	the Sobolev embedding yields
		\begin{align*}
			\mathrm{W}^{2,\frac{4}{3}}(\Omega)^2 \hookrightarrow \mathrm{H}^{1}(\Omega)^2 \hookrightarrow \mathrm{L}^4(\Omega)^2.
		\end{align*}
		By the virtue of interpolation results, we have that, if 
		\begin{align*}
			\v \in \mathrm{L}^\frac{6}{5}(0,T;\mathrm{L}^{4}(\Omega)^2) \  \ \ \text{and} \ \ \ \frac{\partial \v}{\partial t} \in \mathrm{L}^\frac{6}{5}(0,T;\mathrm{L}^\frac{4}{3}(\Omega)^2),
		\end{align*}
		then
		\begin{align*}
			\v \in \C\big([0,T];\big(\mathrm{L}^\frac{4}{3}(\Omega)^2,\mathrm{L}^4(\Omega)^2\big)_{\frac{1}{2},2}\big).
		\end{align*}
		Furthermore, thanks to the Lorentz space (see \cite[Theorem 5.2.1]{JB})
		\begin{align*}
			\big(\mathrm{L}^\frac{4}{3}(\Omega)^2,\mathrm{L}^4(\Omega)^2\big)_{\frac{1}{2},2}=\mathrm{L}^{2,2}(\Omega)^2=\mathrm{L}^{2}(\Omega)^2,
		\end{align*}
		and we deduce that $\v \in \C([0,T];\H)$ and completes the proof.
	\end{proof}
	
	\begin{remark}\label{remrg}
		In light of the preceding regularity results, it is evident that the solution  $(\v, \boldsymbol{\vartheta}^{*})$ to the controllability problem \eqref{cv3} belongs to the space $\mathcal{E}$.
	\end{remark}
	
	\subsection{Local exact controllability to trajectories}\label{subsec 5.3} In this subsection, we prove our main result of local exact controllability to trajectories of the CBF equations \eqref{a1} (Theorem \ref{thmm}) by using the inverse mapping theorem (or epimorphism theorem) given below. In order to achieve this, we  establish the null controllability of the nonlinear system \eqref{a3} which leads to the require result. We mainly follow the ideas outlined in \cite{CGIP,Puel} and references therein.
	\begin{theorem}[Inverse mapping theorem]\label{thmnl}
		Let $\mathcal{E}$ and $\mathcal{G}$ be two Banach spaces and let $\mathcal{A}:\mathcal{E} \to \mathcal{G}$ satisfy $\mathcal{A} \in \C^{1}(\mathcal{E};\mathcal{G})$. Assume that  $\boldsymbol{e}_{0} \in \mathcal{E}$, $\mathcal{A}(\boldsymbol{e}_0)=\boldsymbol{b}_0$,  and $\mathcal{A}'(\boldsymbol{e}_{0}):\mathcal{E}\to \mathcal{G}$ is surjective. Then, there exists $\delta>0$ such that, for every $\boldsymbol{b} \in \mathcal{G}$ satisfying $\|\boldsymbol{b}-\boldsymbol{b}_{0}\|_{\mathcal{G}}< \delta$, there exists a solution of the equation
		\begin{align*}
			\mathcal{A}(\boldsymbol{e})=\boldsymbol{b}, \ \ \ \boldsymbol{e} \in \mathcal{E}.
		\end{align*}
	\end{theorem}
	This theorem can be viewed as a simple corollary of the generalized implicit function theorem, as demonstrated in the work of V. Alekseev, V. Tikhomirov, and S. Fomin \cite{AT}. 
	
	Let us define the operator $\mathcal{A}:\mathcal{E} \to \mathcal{G}$ by
	\begin{align}\label{op}
		\mathcal{A}(\v,\boldsymbol{\vartheta})=\big(\mathcal{L}\v-\mathcal{N}(\v)+\nabla q-\boldsymbol{\vartheta}\mathds{1}_{\omega},\v(0)\big), \ \ \ \text{for all} \ (\v,\boldsymbol{\vartheta}) \in \mathcal{E},
	\end{align}
	where the operators $\mathcal{L}\v$ and $\mathcal{N}(\v)$ are defined in \eqref{a31} and \eqref{a32}, respectively and $\mathcal{E}=\mathcal{E}_{d}$ and $\mathcal{G}=\mathcal{G}_{d}$ for $d=2,3$.\
	
	The next lemma provides the continuity of some multilinear forms which will be required  to apply Theorem \ref{thmnl}.
	\begin{lemma}\label{lemmac}
		Assume that $\wi{\u} \in \mathrm{L}^{\infty}(Q)^d $. Let us consider the bilinear and trilinear forms $\mathcal{N}_{\mathrm{Bi}} : \mathcal{E} \times \mathcal{E} \to \mathcal{W}$ and $\mathcal{N}_{\mathrm{Tri}} : \mathcal{E} \times \mathcal{E} \times \mathcal{E} \to \mathcal{W}$, respectively, given by
		\begin{align}
			\mathcal{N}_{\mathrm{Bi}}(\v_{1},\v_{2})&=  \nabla \cdot(\v_1 \otimes \v_2)+\beta ( \v_1 \cdot \v_2)\wi{\u}+2\beta ( \v_1 \cdot \wi{\u}) \v_2, \label{n1}\\
			\mathcal{N}_{\mathrm{Tri}}(\v_{1},\v_{2},\v_{3})&= \beta ( \v_1 \cdot \v_2) \v_3, \label{n2}
		\end{align}
		where
		\begin{align*}
			\mathcal{W}= \left\{\begin{array}{cc}   \mathrm{L}^2(e^{- s{\psi}}(0,T);\mathrm{H}^{-1}(\Omega)^{2}), &  \ \ \text{ if } d=2,\\
				\mathrm{L}^2(e^{- s{\psi}}(0,T);\mathrm{W}^{-1,6}(\Omega)^{3}), &  \ \text{ if }d=3.\end{array}\right.
		\end{align*}
		Then, there exists $C>0$ such that
		\begin{align}
			\|\mathcal{N}_{\mathrm{Bi}}(\v_{1},\v_{2})\|_{\mathcal{W}} &\leq C \|\v_1\|_{\mathcal{E}}\|\v_2\|_{\mathcal{E}}, \label{n3}\\
			\|\mathcal{N}_{\mathrm{Tri}}(\v_{1},\v_{2},\v_3)\|_{\mathcal{W}}& \leq C \|\v_1\|_{\mathcal{E}}\|\v_2\|_{\mathcal{E}}\|\v_3\|_{\mathcal{E}}.\label{n4}
		\end{align}
	\end{lemma}
	\begin{proof}
		Let us first prove the continuity of the bilinear form $\mathcal{N}_{\mathrm{Bi}}$. In fact,  for $d=2$, we can use $e^{-\frac{ 3s}{4}\hat{\psi}}\v \in \mathrm{L}^4(Q)^2$ for any $(\v,\boldsymbol{\vartheta})\in \mathcal{E}$ and obtain from \eqref{n1} that
		\begin{align*}
			\|\mathcal{N}_{\mathrm{Bi}}(\v_{1},\v_{2})\|_{\mathcal{W}} &\leq  \|\nabla \cdot(\v_1 \otimes \v_2)\|_{\mathrm{L}^2(e^{- s{\psi}}(0,T);\mathrm{H}^{-1}(\Omega)^2)}+ \|\beta ( \v_1 \cdot \v_2)\wi{\u} \|_{\mathrm{L}^2(e^{- s{\psi}}(0,T);\mathrm{H}^{-1}(\Omega)^2)}\\&\quad+ \|2\beta ( \v_1 \cdot \wi{\u}) \v_2\|_{\mathrm{L}^2(e^{- s{\psi}}(0,T);\mathrm{H}^{-1}(\Omega)^2)} \\& \leq C\|\v_1 \otimes \v_2\|_{\mathrm{L}^2(e^{- s{\psi}}(0,T);\H)}+C  \|( \v_1 \cdot \v_2)\wi{\u} \|_{\mathrm{L}^2(e^{- s{\psi}}(0,T);\H)} \\&\quad+ C\|( \v_1 \cdot \wi{\u}) \v_2\|_{\mathrm{L}^2(e^{- s{\psi}}(0,T);\H)}\\&\leq C \|e^{-\frac{ s}{2}{\psi}}\v_1 \|_{\mathrm{L}^4(Q)^2}\|e^{-\frac{ s}{2}{\psi}}\v_2 \|_{\mathrm{L}^4(Q)^2} \\&\quad +
			C \|\wi{\u} \|_{\mathrm{L}^\infty(Q)^2}
			\|e^{-\frac{ s}{2}{\psi}}\v_1 \|_{\mathrm{L}^4(Q)^2}\|e^{-\frac{ s}{2}{\psi}}\v_2 \|_{\mathrm{L}^4(Q)^2}.
		\end{align*}
		From Lemma \ref{lemmacc} (cf. \cite[Lemma 4.1]{Puel}), we also have
		\begin{align}\label{In}
			-\psi \leq - \check{\psi} \leq-\frac{3}{2}\hat{\psi}.
		\end{align}
		Thus, we infer 
		\begin{align*}
			\|\mathcal{N}_{\mathrm{Bi}}(\v_{1},\v_{2})\|_{\mathcal{W}} 
			&\leq C \|e^{-\frac{ 3s}{4}\hat{\psi}}\v_1 \|_{\mathrm{L}^4(Q)^2}\|e^{-\frac{ 3s}{4}\hat{\psi}}\v_2 \|_{\mathrm{L}^4(Q)^2}\\&\quad+	C \|\wi{\u} \|_{\mathrm{L}^\infty(Q)^2}
			\|e^{-\frac{ 3s}{4}\hat{\psi}}\v_1 \|_{\mathrm{L}^4(Q)^2}\|e^{-\frac{ 3s}{4}\hat{\psi}}\v_2 \|_{\mathrm{L}^4(Q)^2},
		\end{align*}
		which leads to \eqref{n3}. 
		
		On the other hand, for $d=3$, we get 
		\begin{align*}
			&	\|\mathcal{N}_{\mathrm{Bi}}(\v_{1},\v_{2})\|_{\mathcal{W}} \\&\leq  \|\nabla \cdot(\v_1 \otimes \v_2)\|_{\mathrm{L}^2(e^{- s{\psi}}(0,T);\mathrm{W}^{-1,6}(\Omega)^3)}+ \big\|\beta ( \v_1 \cdot \v_2)\wi{\u}\big\|_{\mathrm{L}^2(e^{- s{\psi}}(0,T);\mathrm{W}^{-1,6}(\Omega)^3)}\\&\quad+ \big\|2\beta ( \v_1 \cdot \wi{\u})\v_2\big\|_{\mathrm{L}^2(e^{- s{\psi}}(0,T);\mathrm{W}^{-1,6}(\Omega)^3)}\\&\leq  \| \v_1 \otimes \v_2\|_{\mathrm{L}^2(e^{- s{\psi}}(0,T);\mathrm{L}^{6}(\Omega)^3)}+C \big\|(-\Delta)^{-\frac{1}{2}}( \v_1 \cdot \v_2)\wi{\u}\big\|_{\mathrm{L}^2(e^{- s{\psi}}(0,T);\mathrm{L}^{6}(\Omega)^3)}\\&\quad+C \big\|(-\Delta)^{-\frac{1}{2}}( \v_1 \cdot \wi{\u})\v_2\big\|_{\mathrm{L}^2(e^{- s{\psi}}(0,T);\mathrm{L}^{6}(\Omega)^3)} \\& \leq C\|\v_1 \otimes \v_2\|_{\mathrm{L}^2(e^{- s{\psi}}(0,T);\mathrm{L}^{6}(\Omega)^3)}+C  \big\|\nabla(-\Delta)^{-\frac{1}{2}}( \v_1 \cdot \v_2)\wi{\u}\big\|_{\mathrm{L}^2(e^{- s{\psi}}(0,T);\H)}\\&\quad+C  \big\|\nabla(-\Delta)^{-\frac{1}{2}}( \v_1 \cdot \wi{\u})\v_2\big\|_{\mathrm{L}^2(e^{- s{\psi}}(0,T);\H)}
			\\&\leq C \|e^{-\frac{ s}{2}{\psi}}\v_1 \|_{\mathrm{L}^4(0,T;\mathrm{L}^{12}(\Omega)^3)}\|e^{-\frac{ s}{2}{\psi}}\v_2 \|_{\mathrm{L}^4(0,T;\mathrm{L}^{12}(\Omega)^3)} \\&\quad +
			C \|\wi{\u} \|_{\mathrm{L}^\infty(Q)^3}
			\|e^{-\frac{ s}{2}{\psi}}\v_1 \|_{\mathrm{L}^4(Q)^3}\|e^{-\frac{ s}{2}{\psi}}\v_2 \|_{\mathrm{L}^4(Q)^3}
			\\&\leq C \|e^{-\frac{ s}{2}{\psi}}\v_1 \|_{\mathrm{L}^4(0,T;\mathrm{L}^{12}(\Omega)^3)}\|e^{-\frac{ s}{2}{\psi}}\v_2 \|_{\mathrm{L}^4(0,T;\mathrm{L}^{12}(\Omega)^3)} \\&\quad +
			C \|\wi{\u} \|_{\mathrm{L}^\infty(Q)^3}
			\|e^{-\frac{ s}{2}{\psi}}\v_1 \|_{\mathrm{L}^4(0,T;\mathrm{L}^{12}(\Omega)^3)}\|e^{-\frac{ s}{2}{\psi}}\v_2 \|_{\mathrm{L}^4(0,T;\mathrm{L}^{12}(\Omega)^3)}.
		\end{align*}
		By the virtue of the inequality \eqref{In}, we immediately reach at \eqref{n3}. 
		
		Now, we proceed to prove the continuity of the trilinear form $\mathcal{N}_{\mathrm{Tri}}$. For $d=3$, from \eqref{n2}, we deduce that
		\begin{align*}
			&	\|\mathcal{N}_{\mathrm{Tri}}(\v_{1},\v_{2},\v_{3})\|_{\mathcal{W}} \\&=  \|\beta ( \v_1 \cdot \v_2) \v_3\|_{\mathrm{L}^2(e^{- s{\psi}}(0,T);\mathrm{W}^{-1,6}(\Omega)^3)}  \leq C \big\|(-\Delta)^{-\frac{1}{2}} ( \v_1 \cdot \v_2) \v_3\big\|_{\mathrm{L}^2(e^{- s{\psi}}(0,T);\mathrm{L}^{6}(\Omega)^3)}\\&\leq C \big\|\nabla(-\Delta)^{-\frac{1}{2}} ( \v_1 \cdot \v_2) \v_3\big\|_{\mathrm{L}^2(e^{- s{\psi}}(0,T);\H)}\leq C \|e^{- s {\psi}} ( \v_1 \cdot \v_2) \v_3\|_{\mathrm{L}^2(0,T;\H)}  \\&\leq C \|e^{-\frac{ s}{3}{\psi}}\v_1 \|_{\mathrm{L}^6(Q)^3}\|e^{-\frac{ s}{3}{\psi}}\v_2 \|_{\mathrm{L}^6(Q)^3} \|e^{-\frac{ s}{3}{\psi}}\v_3 \|_{\mathrm{L}^6( Q)^3}
			\\&\leq C \|e^{-\frac{ s}{2}\hat{\psi}}\v_1 \|_{\mathrm{L}^6(Q)^3}\|e^{-\frac{ s}{2}\hat{\psi}}\v_2 \|_{\mathrm{L}^6(Q)^3} \|e^{-\frac{ s}{2}\hat{\psi}}\v_3 \|_{\mathrm{L}^6(Q)^3}
			\\&\leq C \|e^{\frac{ s}{4}\hat{\psi}}e^{-\frac{ 3s}{4}{\hat{\psi}}}\v_1 \|_{\mathrm{L}^6(Q)^3}\|e^{\frac{ s}{4}\hat{\psi}}e^{-\frac{ 3s}{4}{\hat{\psi}}}\v_2 \|_{\mathrm{L}^6(Q)^3} \|e^{\frac{ s}{4}\hat{\psi}}e^{-\frac{ 3s}{4}{\hat{\psi}}}\v_3 \|_{\mathrm{L}^6(Q)^3},
		\end{align*}
		and \eqref{n4} follows by using \eqref{In}. Similarly, for $d=2$, we have
		\begin{align*}
			&	\|\mathcal{N}_{\mathrm{Tri}}(\v_{1},\v_{2},\v_{3})\|_{\mathcal{W}}^2 =\int_0^T \|e^{- s{\psi}}\beta ( \v_1 \cdot \v_2) \v_3\|_{\mathrm{H}^{-1}(\Omega)^2}^2\d t \leq C \int_0^T \|e^{- s{\psi}} ( \v_1 \cdot \v_2) \v_3\|_{\mathrm{L}^\frac{12}{11}(\Omega)^2}^2\d t
			\\& \leq C\int_0^T \|e^{-\frac{ s}{3}{\psi}}\v_1\|_{\mathrm{L}^3(\Omega)^2}^2\|e^{-\frac{ s}{3}{\psi}}\v_2\|_{\mathrm{L}^3(\Omega)^2}^2 \|e^{-\frac{ s}{3}{\psi}}\v_3\|_{\mathrm{L}^4(\Omega)^2}^2 \d t
			\\& \leq C \sup_{t\in[0,T]}\left(\|e^{-\frac{ s}{3}{\psi}}\v_1\|_{\H}^\frac{4}{3}\|e^{-\frac{ s}{3}{\psi}}\v_2\|_{\H}^\frac{4}{3}\right)\bigg(\int_0^T \|e^{-\frac{ s}{3}{\psi}}\nabla\v_1\|_{\H}^\frac{2}{3}\|e^{-\frac{ s}{3}{\psi}}\nabla\v_2\|_{\H}^\frac{2}{3} \|e^{-\frac{ s}{3}{\psi}}\v_3\|_{\mathrm{L}^4(\Omega)^2}^2 \d t\bigg)
			\\& \leq C \sup_{t\in[0,T]}\left(\|e^{-\frac{ s}{3}{\psi}}\v_1\|_{\H}^\frac{4}{3}\|e^{-\frac{ s}{3}{\psi}}\v_2\|_{\H}^\frac{4}{3}\right)
			\bigg(\int_0^T \|e^{-\frac{ s}{3}{\psi}}\nabla\v_1\|_{\H}^2 \d t\bigg)^\frac{1}{3}	\\& \quad\times\bigg(\int_0^T \|e^{-\frac{ s}{3}{\psi}}\nabla\v_2\|_{\H}^2 \d t\bigg)^\frac{1}{3}\bigg(\int_0^T \|e^{-\frac{ s}{3}{\psi}}\v_3\|_{\mathrm{L}^4}^6 \d t\bigg)^\frac{1}{3}
			\\&\leq C \sup_{t\in[0,T]}\left(\|e^{-\frac{ s}{2}\hat{\psi}}\v_1\|_{\H}^\frac{4}{3}\|e^{-\frac{ s}{2}\hat{\psi}}\v_2\|_{\H}^\frac{4}{3}\right)
			\|e^{-\frac{ s}{2}\hat{\psi}}\v_1 \|_{\mathrm{L}^2(0,T;\V)}^\frac{2}{3}	\\& \quad\times\|e^{-\frac{ s}{2}\hat{\psi}}\v_2 \|_{\mathrm{L}^2(0,T;\V)}^\frac{2}{3} \|e^{-\frac{ s}{2}\hat{\psi}}\v_3 \|_{\mathrm{L}^6( 0,T;\mathrm{L}^{4}(\Omega)^2)}^2
			\\&\leq C \sup_{t\in[0,T]}\left(\|e^{\frac{ s}{4}\hat{\psi}}e^{-\frac{ 3s}{4}{\hat{\psi}}}\v_1\|_{\H}^\frac{4}{3}\|e^{\frac{ s}{4}\hat{\psi}}e^{-\frac{ 3s}{4}{\hat{\psi}}}\v_2\|_{\H}^\frac{4}{3}\right)
			\|e^{\frac{ s}{4}\hat{\psi}}e^{-\frac{ 3s}{4}{\hat{\psi}}}\v_1 \|_{\mathrm{L}^2(0,T;\V)}^\frac{2}{3}	\\& \quad\times\|e^{\frac{ s}{4}\hat{\psi}}e^{-\frac{ 3s}{4}{\hat{\psi}}}\v_2 \|_{\mathrm{L}^2(0,T;\V)}^\frac{2}{3} \|e^{\frac{ s}{4}\hat{\psi}}e^{-\frac{ 3s}{4}{\hat{\psi}}}\v_3 \|_{\mathrm{L}^6( 0,T;\mathrm{L}^{4}(\Omega)^2)}^2,
		\end{align*}
		where we have used the embedding $\mathrm{H}^1(\Omega)^2 \hookrightarrow \mathrm{L}^{12}_{\sigma}(\Omega)^2 \hookrightarrow \mathrm{L}^\frac{12}{11}_{\sigma}(\Omega)^2	\hookrightarrow \mathrm{H}^{-1}(\Omega)^2$, Gagliardo-Nirenberg's and H\"older's inequalities. Using \eqref{In}, we deduce \eqref{n4}, and the proof of Lemma \ref{lemmac} is completed.
	\end{proof}
	\begin{lemma}\label{lemmac1}
		Let us assume that $\wi{\u} \in \mathrm{L}^\infty(Q)^d$. Then, $\mathcal{A} \in \C^1(\mathcal{E};\mathcal{G})$.
	\end{lemma}
	\begin{proof}
		From \eqref{n1} and \eqref{n2}, we can see that the operator $\mathcal{N}(\v)$ defined in \eqref{a32} can be written as $$\mathcal{N}(\v)=	\mathcal{N}_{\mathrm{Bi}}(\v,\v)+	\mathcal{N}_{\mathrm{Tri}}(\v,\v,\v).$$ Thus, \eqref{n3} and \eqref{n4} imply that $$\| \mathcal{N} (\v) \|_{\mathcal{G}} \leq  C(\|\v\|_{\mathcal{E}}^2+\|\v\|_{\mathcal{E}}^3).$$ Therefore, we infer that $\mathcal{A}:\mathcal{E} \to \mathcal{G}$ defined in \eqref{op} is well-defined. 
		
		The first component of the operator $\mathcal{A}(\v,\boldsymbol{\vartheta})$ is linear (and consequently $\C^1$), except $\mathcal{N} (\v)$. However, the bilinear and trilinear forms given by \eqref{n1} and \eqref{n2}, respectively, are continuous (cf. Lemma \ref{lemmac}) and consequently, we deduce that $\mathcal{A}$ is a $\C^1$ map from $\mathcal{E}$ to $\mathcal{G}$, which proves the lemma.
	\end{proof}
	As a consequence of this lemma, we can prove our main result (Theorem \ref{thmm}). For the proof, we follow the strategy introduced in \cite{CGIP,Im1,Puel}, etc.
	\begin{proof}[Proof of Theorem \ref{thmm}]
		Let us take $\boldsymbol{e}_0=(\mathbf{0},\mathbf{0}) \in \mathcal{E}$ and $\boldsymbol{b}_0=(\mathbf{0},\mathbf{0}) \in \mathcal{G}$. It follows that $\mathcal{A}'(\boldsymbol{e}_{0}):\mathcal{E}\to \mathcal{G}$ is given by  
		\begin{align*}\mathcal{A}'(\boldsymbol{e}_{0})[\v_{*}, \boldsymbol{\vartheta}_{*}]=\big(\mathcal{L}\v_{*}+\nabla q_{*}-\boldsymbol{\vartheta}_{*}\mathds{1}_{\omega},\v_{*}(0)\big), \ \ \ \text{for all} \ (\v_{*}, \boldsymbol{\vartheta}_{*}) \in \mathcal{E}, 
		\end{align*}
		and $\mathcal{A}'(\boldsymbol{e}_{0})$ is a surjective linear map from $\mathcal{E}$ to $\mathcal{G}$, as implied by the null controllability result for the linearized system \eqref{cv3} (see also Remark \ref{remrg}), or equivalently, $\mathrm{Im}(\mathcal{A}'(\boldsymbol{e}_{0}))=\mathcal{G}$.  An application of the inverse mapping theorem (Theorem \ref{thmnl}) gives exactly the conclusions of  Theorem \ref{thmm} if we take $\boldsymbol{b}=(\mathbf{0},\v_0) \in \mathcal{G}$, that is, it gives the existence of $\delta>0$ such that, if $$\|\v_0\|_{\mathrm{L}^{2d-2}(\Omega)^d} \leq \delta,$$ then we can find a control $\boldsymbol{\vartheta}$ such that the associated solution to \eqref{a3} satisfies $\big(\v,\boldsymbol{\vartheta} \big)\in\mathcal{E}$ and 
		\begin{align*}
			\v(\cdot,T)=\mathbf{0}, \ \ \  \text{in}  \ \Omega,
		\end{align*}
		and	we reach at the desired result.	Thus the proof of Theorem \ref{thmm} is completed.
	\end{proof}
	\begin{remark}
		We could have taken in the equation for $\wi{\u}$ an external force $\wi{\f}$ different from $\f$  (cf. \eqref{a2}).  We can set $$\f=\wi{\f}+\f^{*},$$
		where $\wi{\xi}^{-\frac{3}{2}}e^{- s\wi{\psi}} \wi{\f}$ is sufficiently  small and finite, and obtain the local exact controllability to trajectories of the CBF equations \eqref{a1}.
	\end{remark}
	\medskip\noindent
	{\bf Acknowledgments:} P. Kumar and M. T. Mohan would  like to thank the Department of Science and Technology (DST), India for Innovation in Science Pursuit for Inspired Research (INSPIRE) Faculty Award (IFA17-MA110).

\begin{thebibliography}{99}
		\bibitem{AT} V.M.  Alekseev, V.M. Tikhomirov and S.V. Fomin, \emph{Optimal control}, Contemporary Soviet Mathematics, Consultants Bureau, New york, 1987.
		
		\bibitem{EVA} E.V. Amosova, Exact local controllability for the equations of viscous gas dynamics, \emph{Differ. Equ.}, \textbf{47} (2011), 1776--1795.
		
		
		\bibitem{BES} M. Badra, S. Ervedoza and S. Guerrero, Local controllability to trajectories for non-homogeneous incompressible Navier-Stokes equations, \emph{Ann. Inst. H. Poincar\'e C Anal. Non Lin\'eaire}, \textbf{33} (2016), 529--574.
		
		\bibitem{AB} A. Bejan and D. Nield, \emph{Convection in porous media}, Berlin: Springer, 2006. 
		
		\bibitem{JB} J. Bergh and J. L\"ofstr\"om, \emph{Interpolation spaces}, Springere Verlag, Newe, 1976.
		
		\bibitem{MGB} M.G. Burgos, S. Guerrero and J.-P. Puel,
		Local exact controllability to the trajectories of the Boussinesq system via a fictitious control on the divergence equation, \emph{Commun. Pure Appl. Anal.}, \textbf{8} (2009), 311--333. 	
		
		\bibitem{FC} E. Fern\'andez-Cara, On the approximate and null controllability of the Navier-Stokes equations, \emph{SIAM Rev.}, \textbf{41} (1999), 269--277. 
		
		\bibitem{CGIP} E. Fern\'andez-Cara, S. Guerrero, O.Y. Imanuvilov and J.-P. Puel, Local exact controllability to the trajectories of the Navier-Stokes equations, \emph{J. Math. Pures Appl.}, \textbf{83} (2004), 1501--1542.
		
		\bibitem{EFC}	E. Fern\'andez-Cara and S. Guerrero, Global Carleman inequalities for parabolic systems and applications to controllability, \emph{SIAM J. Control Optim.}, \textbf{45} (2006), 1399--1446.
		
		\bibitem{EFC1}	E. Fern\'andez-Cara, S. Guerrero, O.Y. Imanuvilov and J.-P. Puel,	Some controllability results for the $N$-dimensional Navier-Stokes and Boussinesq systems with $N-1$ scalar controls, \emph{SIAM J. Control Optim.}, \textbf{45} (2006), 146--173.
		
		\bibitem{NC} N. Carre\~no and S. Guerrero,
		Local null controllability of the $N$-dimensional Navier-Stokes system with $N-1$ scalar controls in an arbitrary control domain, \emph{J. Math. Fluid Mech.}, \textbf{15} (2013),  139--153.
		
		\bibitem{EC} E. Cerpa and A. Mercado, Local exact controllability to the trajectories of the 1D Kuramoto-Sivashinsky equation, \emph{J. Differential Equations}, \textbf{250} (2011), 2024--2044.
		
		
		
		
		
		\bibitem{JMC} J.-M. Coron, On the controllability of the 2-D incompressible Navier-Stokes equations with
		the Navier slip boundary conditions, \emph{ESAIM Control Optim. Calc. Var.}, \textbf{1} (1996), 35--75.
		
		\bibitem{JMC1} J.-M. Coron, On the controllability of 2-D incompressible perfect fluids, \emph{J. Math. Pures Appl.}, \textbf{75} (1996), 155--188.
		
		\bibitem{JMC2} J.-M. Coron and A.V. Fursikov, Global exact controllability of the 2-D Navier-Stokes equations on manifold without boundary, \emph{J. Russian Math. Phys.}, \textbf{4} (1996), 1--20.
		
		
		
		
		\bibitem{Evans} L.C. Evans, \emph{Partial Differential Equations}, Grad. Stud. Math., vol. 19, Second Edition, Amer. Math. Soc., Providence,	RI, 2010.
		
		
		
		\bibitem{CF} C. Fabre,	Uniqueness results for Stokes equations and their consequences in linear and nonlinear control problems, \emph{ESAIM Control Optim. Calc. Var.}, \textbf{1} (1996), 267--302.
		
		
		
		\bibitem{CLF} 	C.L. Fefferman, K.W. Hajduk and J.C. Robinson,	Simultaneous approximation in Lebesgue and Sobolev norms via eigenspaces, \emph{Proc. Lond. Math. Soc. (3)}, {\bf 125}(4) (2022),  759--777. 
		
		\bibitem{Avf}	A.V. Fursikov, Exact boundary zero controllability of three-dimensional Navier-Stokes
		equations, \emph{J. Dyn. Control Syst.}, \textbf{1} (1995), 325--350.
		
		\bibitem{FYIm}	A.V. Fursikov and O.Y. Imanuvilov, On approximate controllability of the Stokes system, \emph{Ann. Fac. Sci. Toulouse Math. (6)}, \textbf{2} (1993), 205--232.
		
		\bibitem{FYIm1}	A.V. Fursikov and O.Y. Imanuvilov, On exact boundary zero-controlability of two-dimensional Navier-Stokes equations, \emph{Acta Appl. Math.}, \textbf{37} (1994), 67--76.
		
		\bibitem{FIm} A.V. Fursikov and O.Y. Imanuvilov, \emph{Controllability of Evolution Equations}, Lecture Note
		Series 34, Research Institute of Mathematics, Seoul National University, Seoul, 1996.
		
		\bibitem{FIm 1} A.V. Fursikov and O.Y. Imanuvilov,	Local exact controllability of the two-dimensional Navier-Stokes equations, \emph{Sb. Math.}, 1996.
		
		\bibitem{FIm 2} A.V. Fursikov and O.Y. Imanuvilov, Local exact boundary controllability of the Boussinesq equation, \emph{SIAM J. Control Optim.}, \textbf{36} (1998), 391--421.
		
		\bibitem{AVIm}	A.V. Fursikov and O.Y. Imanuvilov, Exact controllability of the Navier-Stokes and Boussinesq equations, \emph{Russian Math. Surveys}, \textbf{54} (1999), 565--618.
		
		\bibitem{PGAO} P. Gao, Local exact controllability to the trajectories of the Swift-Hohenberg equation, \emph{Nonlinear Anal.}, \textbf{139} (2016), 169--195.
		
		\bibitem{GS} Y. Giga and H. Sohr, Abstract $\mathrm{L}^p$ estimates for the Cauchy problem with applications to the Navier-Stokes equations in exterior domains, \emph{J. Funct. Anal.}, \textbf{102} (1991), 72--94.
		
		
		\bibitem{SG} S. Guerrero, Local exact controllability to the trajectories of the Boussinesq system, \emph{Ann. Inst. H. Poincar\'e C Anal. Non Lin\'eaire}, \textbf{23} (2006), 29--61.
		
		\bibitem{SG1} S. Guerrero, Local exact controllability to the trajectories of the Navier-Stokes system with nonlinear Navier-slip boundary conditions,
		\emph{ESAIM Control Optim. Calc. Var.}, \textbf{12} (2006),  484--544.
		
		\bibitem{SG2} S. Guerrero, O.Y. Imanuvilov and J.-P. Puel,
		A result concerning the global approximate controllability of the Navier-Stokes system in dimension $3$, \emph{J. Math. Pures Appl.}, \textbf{98} (2012), 689--709.
		
		\bibitem{PG} P. Guzm\'an, Local exact controllability to the trajectories of the Cahn-Hilliard equation, \emph{Appl. Math. Optim.}, \textbf{82} (2020), 279--306.
		
		
		
		
		\bibitem{Im2} O.Y. Imanuvilov, Local exact controllability for the 2D Navier-Stokes equations with the Navier slip boundary conditions, \emph{Turbulence Modeling and Vortex Dynamics,} Lecture Notes in Physics, \textbf{491} (1997), 148--168.
		
		\bibitem{Im21} O.Y. Imanuvilov, Local exact controllability for the 2-D Boussinesq equations with the Navier slip boundary conditions, \emph{ESAIM: Proceedings}, 1998,  153--170.
		
		\bibitem{Im3} O.Y. Imanuvilov, On exact controllability for the Navier-Stokes equations, \emph{ESAIM Control Optim. Calc. Var.}, \textbf{3} (1998), 97--131. 
		
		\bibitem{Im1} O.Y. Imanuvilov, Remarks on exact controllability for the Navier-Stokes equations, \emph{ESAIM Control Optim. Calc. Var.}, \textbf{6} (2001), 39--72.
		
		\bibitem{ImP} O.Y. Imanuvilov and J.-P. Puel, Global Carleman estimates for weak elliptic non homogeneous Dirichlet problem, \emph{Int. Math. Res. Not. IMRN}, \textbf{16} (2003), 883--913.
		
		\bibitem{ImpY} O.Y. Imanuvilov, J.-P. Puel and M. Yamamoto, Carleman estimates for parabolic equations with non homogeneous boundary conditions, \emph{Chin. Ann. Math. Ser. B}, \textbf{30} (2009), 333--378.
		
		\bibitem{ImpY1} O.Y. Imanuvilov, J.-P. Puel and M. Yamamoto, Carleman estimates for second order non homogeneous parabolic equations, To appear.
		
		
		
		
		
		
		
		
		
		
		
		
		
		
		\bibitem{MTM} M.T. Mohan, On the convective Brinkman-Forchheimer  equations, \emph{Submitted}. 
		\bibitem{MTM1} M.T. Mohan, Stochastic convective Brinkman-Forchheimer equations, \emph{Submitted}, \url{https://arxiv.org/pdf/2007.09376.pdf}. 
		
		
		\bibitem{Puel} J.-P. Puel, \emph{Optimization with PDE Constraints}, Lecture Notes in Computational Science and Engineering \textbf{101}, Springer, 2014.
		
		
		
		\bibitem{AS} A. Shirikyan, Approximate controllability of three-dimensional Navier-Stokes equations, \emph{Commun. Math. Phys.}, 	\textbf{266} (2006), 123--151.
		
		 \bibitem{VAS} V.A. Solonnikov,  Estimates of the solution of a certain initial-boundary value problem for a linear nonstationary system of Navier-Stokes equations, \emph{J. Soviet Math.}, {\bf 8} (1977), 467--523. 
		
		\bibitem{LT} L. Tarter, \emph{An Introduction to Sobolev
			Spaces and Interpolation Spaces}, Springer, Berlin/Heidelberg/New York, 2007.
			
		\bibitem{RTR} 	R. Triggiani, A note on the lack of exact controllability for mild solutions in Banach spaces, \emph{SIAM J. Control Optim.}, {\bf  15}(3) (1977),  407--411. 
		
		\bibitem{RT} R. Temam,  \emph{Navier-Stokes Equations, Theory and Numerical Analysis}, North-Holland, Amsterdam, 1984.
		
		
	\end{thebibliography}
\end{document}